\documentclass{amsart}

\usepackage{amssymb}
\usepackage{latexsym}
\usepackage{amsmath}
\usepackage{euscript}

\newtheorem{thm}{Theorem}[section]
\newtheorem{cor}[thm]{Corollary}
\newtheorem{lem}[thm]{Lemma}
\newtheorem{prop}[thm]{Proposition}
\theoremstyle{definition}
\newtheorem{defn}[thm]{Definition}
\theoremstyle{remark}
\newtheorem{rem}[thm]{Remark}
\newtheorem*{ex}{Example}
\numberwithin{equation}{section}

\def\sS{{\mathfrak S}}

      \def\dC{{\mathbb C}}

   \def\dN{{\mathbb N}}   
      \def\dR{{\mathbb R}}

   \def\dZ{{\mathbb Z}}

      \def\cI{{\mathcal I}}
      
\def\cM{{\mathcal M}}      
      
   \def\cT{{\mathcal T}}

\def\wt#1{{{\widetilde #1} }}

\def\wh#1{{{\widehat #1} }}

\def\bm\chi{\mbox{\boldmath$\chi$}}

\def\IM{{\rm Im\,}}
\def\ran{{\rm ran\,}}

\def\rank{{\rm rank\,}}

\let\xker=\ker \def\ker{{\xker\,}}

\def\sgn{{\rm sgn\,}}

\def\cmr{{\dC \setminus \dR}}

\newcommand{\adots}{\makebox[0.9ex][l]{\raisebox{-0.2ex}{.}} \raisebox{0.4ex}{.}
\makebox[0.9ex][r]{\raisebox{1.1ex}{.}}}

\begin{document}

\title[Truncated moment problems]
{Truncated moment problems in the class of generalized Nevanlinna
functions}

\author[V.A.~Derkach]{Vladimir Derkach}
\address{Department of Mathematical Analysis \\
Donetsk National University \\
Universitetskaya str. 24 \\
83055 Donetsk \\
Ukraine} \email{derkach.v@gmail.com}
\author[S.~Hassi]{Seppo Hassi}
\address{Department of Mathematics and Statistics \\
University of Vaasa \\
P.O. Box 700, 65101 Vaasa \\
Finland} \email{sha@uwasa.fi}
\author[H.S.V.~de~Snoo]{Henk de Snoo}
\address{Department of Mathematics and Computing Science \\
University of Groningen \\
P.O. Box 800, 9700 AV Groningen \\
Nederland} \email{desnoo@math.rug.nl}


\thanks{The research was supported by NWO, the Finnish Cultural Foundation, South
Ostrobothnia Regional fund.} 


\begin{abstract}
Truncated moment problems in the class of generalized Nevanlinna
functions are investigated. General solvability criteria will be
established, covering both the even and odd problems, including
complete parametrizations of solutions. The main new results concern
the case where the corresponding Hankel matrix of moments is
degenerate. One of the new effects which reveals in the indefinite
case is that the degenerated moment problem may have infinitely many
solutions. However, with a careful application of an indefinite
analogue of a step-by-step Schur algorithm a complete description of
the set of solutions will be obtained.
\end{abstract}

\maketitle
\section{Introduction}

The main purpose of this paper is to study general truncated (real)
moment problems and some associated interpolation problems involving
a finite sequence of real numbers $s_0,s_1, \dots, s_{\ell}$. In
order to describe some of the contents and results in the paper it
is natural to start by recalling a couple of notions and results
appearing in classical truncated moment problems.

The truncated Hamburger moment problem for real numbers $s_0,s_1,
\dots, s_{\ell}$ $(\ell\in\dZ_+)$ consists of finding a positive
measure $\mu$ on $I=\dR$ for which
\begin{equation}\label{0.1}
 \int_I t^j\,d\mu(t)=s_j, \quad j=0,1,\dots, \ell.
\end{equation}
This problem will be called {\it odd} or {\it even }, if the number
$\ell$ is odd or even, respectively. In the case where $I=\dR_+$ the
problem~\eqref{0.1}  is called the truncated Stieltjes moment
problem. Due to the Hamburger-Nevanlinna theorem \cite[Theorem
3.2.1]{Ach61} in the even case ($\ell=2n$) the conditions in
\eqref{0.1} can be rewritten in terms of the associated function
$\varphi(\lambda)$ defined by
\begin{equation}\label{eq:AssF}
     \varphi(\lambda)=\int_I \frac{d\mu(t)}{t-\lambda}, \quad
\lambda \in \cmr,
\end{equation}
as the following interpolation problem at $\lambda=\infty$:
\begin{equation}\label{0.2}
\varphi(\lambda)=-\frac{s_0}{\lambda}-\dots-\frac{s_{\ell}}{\lambda^{\ell+1}}
+o\left(\frac{1}{\lambda^{\ell+1}}\right), \quad \lambda \wh \to
\infty.
\end{equation}
The notation $\lambda \wh \to \infty$ means that $\lambda \to
\infty$ nontangentially, i.e. $\delta < \arg \lambda < \pi-\delta$
for some $\delta > 0$. Recall, that $\varphi$ belongs to the class
$\mathbf{N}_0$ of Nevanlinna functions, i.e., $\varphi(\lambda)$ is
holomorphic on $\cmr$, satisfies the symmetry condition
$\overline{\varphi(\lambda)}=\varphi(\bar{\lambda})$, and has a
nonnegative imaginary part for all $\lambda \in \dC_+$. The moment
problem \eqref{0.1} can now be reformulated as follows: find a
Nevanlinna function $\varphi(\lambda)$ such that \eqref{0.2} holds.
It follows easily from \eqref{0.1} that the following inequality
\begin{equation}\label{0.3}
S_n:=(s_{i+j})_{i,j=0}^n \ge 0,
\end{equation}
is necessary for the problem \eqref{0.1} to be solvable. In the case
where the matrix $S_n$ is invertible this condition is also
sufficient for \eqref{0.1} to be solvable, and all its solutions are
described by the formula (see \cite{Nev29}):
\begin{equation}\label{0.4}
\varphi(\lambda)=\int_\dR \frac{d\mu(t)}{t-\lambda}
=-\frac{{Q_{n}}(\lambda)\tau(\lambda)+Q_{n+1}(\lambda)}
{{P_{n}}(\lambda)\tau(\lambda)+P_{n+1}(\lambda)},
\end{equation}
where $P_n$ are polynomials of the first kind orthonormal with
respect to ${\mathfrak S}$,
\begin{equation}\label{0.04}
 Q_{n}(\lambda)={\mathfrak S} \left(
 \frac{P_{n}(t)-P_{n}(\lambda)}{t-\lambda} \right)
\end{equation}
are polynomials of the second kind and $\tau(\lambda)$ is an
arbitrary Nevanlinna function from the class $\mathbf{N_0}$ which
satisfies the Nevanlinna condition
\begin{enumerate}
  \item [(E)] \qquad$\tau(\lambda)=o(\lambda)\quad\mbox{ as }\quad\lambda \wh \to
  \infty$;
\end{enumerate}
in \eqref{0.04} $\sS$ stands for the nonnegative functional defined
on the set $\dC[t]$ of polynomials via
\[
  \sS(t^j)=s_j,\quad j=0,1,\dots,2n.
\]

Notice that this classical result depends essentially also on the
assumption that the moment problem is even (i.e., $\ell=2n$).
Indeed, the odd Hamburger moment problem is not equivalent to the
interpolation problem~\eqref{0.2}. A convenient framework to
formulate the problem in the odd case is provided by the classes
${\mathbf N}_{0,-\ell}$ appearing in \cite{HSW}: they consist of
functions $f\in{\mathbf N}_0$ of the form~\eqref{eq:AssF}, such that
the measure $\mu$ in \eqref{eq:AssF} satisfies the condition
\[
\int_{\dR}(1+|t|^\ell)d\mu(t)<\infty\quad(\ell=-1,0,1,2,\dots).
\]
Then the Hamburger-Nevanlinna theorem can be restated as follows:
$\mu$ is a solution of the moment problem \eqref{0.1} if and only if
the associated function $\varphi$ belongs to the class ${\mathbf
N}_{0,-\ell}$ and has the asymptotic expansion~\eqref{0.2}. It is a
consequence of the results in the present paper (see
Corollary~\ref{cor:5.1}) that the set of solutions of the
nondegenerate odd moment problem \eqref{0.1} can also be given in
the form~\eqref{0.4}, where $\tau$ now ranges over the class
${\mathbf N}_{0,1}$ and satisfies
\begin{enumerate}
  \item [(O)]
  \qquad$\tau(\lambda)=o(1)\quad\mbox{ as }\quad\lambda \wh \to \infty.$
\end{enumerate}
It should be mentioned that in the special case of the nondegenerate
odd Stieltjes moment problem the set of solutions was parametrized
in~\cite{Kr67} with the parameter $\tau$ ranging over the class of
Stieltjes functions (i.e. Nevanlinna functions of the
form~\eqref{eq:AssF} with $I=\dR_+$). Such functions automatically
belong to the class ${\mathbf N}_{0,1}$ and they satisfy also the
condition~(O). However, it seems to the authors that for general
measures, whose support is not contained in some semiaxis in $\dR$,
the above mentioned description of the solution set for
nondegenerate odd Hamburger moment problem has not appeared in the
literature earlier.

In the case where the matrix $S_n$ is degenerate the condition
\eqref{0.3} is not anymore sufficient for the problem \eqref{0.1} to
be solvable; see \cite{KN77,Ioh82,CF91}. Recall that the
\textit{Hankel rank}, denoted by $\rank ({\bf s},2n)$, of the
sequence $({\bf s},2n)=\{s_j\}_{j=0}^{2n}$ is defined as follows:
$\rank({\bf s},2n)=n+1$ if $\det S_n\neq 0$, otherwise, $\rank ({\bf
s},2n)$ is the smallest integer $r$, $0\le r\le n$, such that
\begin{equation}\label{eq:vect}
 \begin{pmatrix} s_r \\ \vdots\\s_{r+n} \end{pmatrix}
 \in \mbox{span }\left( \begin{pmatrix} s_0 \\ \vdots\\s_{n} \end{pmatrix}, \dots,
  \begin{pmatrix} s_{r-1} \\ \vdots\\s_{r-1+n} \end{pmatrix}\right).
\end{equation}
 By a Frobenius theorem (see~\cite[Lemma X.10.1]{Ga})
Hankel rank  of $({\bf s},2n)$ is the smallest integer $r$, $1\le
r\le n+1$, such that
\[
\det S_{r-1}\ne 0, \mbox{ and }\det S_{j}= 0 \mbox{ for }j\ge r.
\]
In particular, $\rank ({\bf s},2n)=0$ if $s_0=\dots=s_n=0$,
otherwise $\rank ({\bf s},2n)$ is the smallest integer $r$, $1\le
r\le n+1$, such that
\[
\det S_{r-1}\ne 0, \mbox{ and }\det S_{j}= 0 \mbox{ for }j\ge r.
\]

 A sequence
$({\bf s},2n)=\{s_j\}_{j=0}^{2n}$ with the Hankel rank $r=\mbox{rank
}({\bf s},2n)$ is called {\it recursively generated}, if there exist
numbers $\alpha_0,\dots,\alpha_{r-1}$, such that
   \begin{equation} \label{0.01}
s_{j} =\alpha_0s_{j-r}+\dots +\alpha_{r-1}s_{j-1}\quad (r\le j\le
2n).
   \end{equation}

Now solvability criteria for the degenerate truncated moment
\eqref{0.1} can be formulated as follows.

\begin{thm}[\cite{CF91}]\label{een}
Let the matrix $S_n=(s_{i+j})_{i,j=0}^n$ be nonnegative and
degenerate, and let $r=\rank{\bf s}$. Then the following statements
are equivalent:
\begin{enumerate}\def\labelenumi{\rm (\roman{enumi})}

\item the moment problem \eqref{0.1} is solvable;

\item $\rank S_n=r$;

\item $S_n$ admits a nonnegative Hankel extension $S_{n+1}$;

\item the sequence $({\bf s},2n)=\{s_j\}_{j=0}^{2n}$ is recursively
generated.
\end{enumerate}
If any of the assumptions (i)--(iv) are satisfied, then the problem
\eqref{0.1} has a unique solution
$\varphi(\lambda)=-\frac{Q_{r}(\lambda)}{P_{r}(\lambda)}$.
\end{thm}

It is interesting to note that Theorem~\ref{een} contains as a
corollary the following rigidity result due to D. Burns and S.
Krantz~\cite{BKr94}: if $\varphi$ is a rational Nevanlinna function
of degree $r$ with the asymptotic expansion~\eqref{0.2} and $\psi$
is a Nevanlinna function such that
$\psi(\lambda)=\varphi(\lambda)+o(\frac{1}{\lambda^{2r+1}})$ as
$\lambda \wh \to \infty$, then
$\psi(\lambda)\equiv\varphi(\lambda)$.

The main subject of the present paper is the study of degenerate odd
and even moment problems involving finite sequences $s_0,s_1, \dots,
s_{\ell}$ of real numbers by means of functions belonging to the
class of generalized Nevanlinna functions, which contains the class
of Nevanlinna functions appearing in \eqref{eq:AssF} as a subclass.

\begin{defn}\label{def:1.1}{\rm (\cite{KL79})}
Let $\kappa \in \dN$. A function $\varphi$  meromorphic on $\dC_+$
is said to be from the class $\mathbf{N}_\kappa$, $\kappa \in \dN$,
of \textit{generalized Nevanlinna functions} with $\kappa$ negative
squares, if the kernel
\[
\mathsf{N}_\omega(\lambda)=\frac{\varphi(\lambda)-\overline{\varphi(\omega)}}{\lambda-\bar{\omega}}
\]
has $\kappa$ negative squares on $\dC_+$, i.e. for every choice of
$m\in\dN$, $\lambda_1, \dots, \lambda_m \in \dC_+$ the matrix
\[
 (\mathsf{N}_{\lambda_k}(\lambda_i))_{i,k=1}^m
\]
has at most $\kappa$ and for some choice of $m,\,\lambda_j$ exactly
$\kappa$ negative eigenvalues.
\end{defn}
In \cite{DHS1,DHS07} (see Definitions~\ref{def:Nk1}, \ref{def:Nkl}
below) subclasses $\mathbf{N}_{\kappa,-\ell}$ of the class
$\mathbf{N}_\kappa$ were introduced as indefinite analogues of the
subclasses $\mathbf{N}_{0,-\ell}$ appearing in \cite{HSW}.

In this paper we consider in a parallel way the following two
problems:

\noindent \textbf{Indefinite truncated moment problem ${
MP}_\kappa(s,\ell)$}: Given are $\kappa\in \dN,\,\ell \in \dZ_+$,
and $s_0, \dots, s_{\ell}\in\dR$. Find  a function $\varphi \in
\mathbf{N}_{\kappa,-\ell}$ with the asymptotic
expansion~\eqref{0.2}. Denote by $\cM_\kappa(s,\ell)$ the set of
solutions of this problem.

\noindent \textbf{Multiple indefinite interpolation problem ${
IP}_\kappa(s,\ell)$}: Given are $\kappa\in \dN,\,\ell \in \dZ_+$,
and $s_0, \dots, s_{\ell}\in\dR$. Find  a function $\varphi \in
\mathbf{N}_\kappa$ with the asymptotic expansion~\eqref{0.2}. The
set of functions with these properties is denoted by
$\cI_\kappa(s,\ell)$.

As was mentioned above, the problem ${MP}_0(s,\ell)$ is equivalent
to the truncated Hamburger moment problem~\eqref{0.1}. Furthermore,
in the even case ($\ell=2n$)
$\cM_\kappa(s,\ell)=\cI_\kappa(s,\ell)$, while in the odd case
($\ell=2n+1$) we have $\cM_\kappa(s,\ell)\subset
\cI_\kappa(s,\ell)$, but the reverse inclusion fails to hold in
general.

The problems ${ MP}_\kappa(s,\ell)$, ${ IP}_\kappa(s,\ell)$ will be
called nondegenerate, if
\begin{equation}\label{0.7}
\det S_n\ne 0, \mbox{ for }n=[\ell/2];
\end{equation}
otherwise they are called degenerate. The indices $j$ for which
$\det S_{j-1} \ne 0$ are called \textit{normal indices} of the
Hankel matrix $S_n$. Let
\[
 n_1<n_2<\dots<n_N \le n+1
\]
be the sequence of all \textit{normal indices} of the matrix $S_n$.
In the case of arbitrary Hankel matrix $S_n$ we show that the
largest normal index $n_N$ of $S_n$ coincides with the Hankel rank
of the sequence ${\bf s}$.

A necessary condition for the problems ${ MP}_\kappa({\bf s},\ell)$,
${ IP}_\kappa(s,\ell)$ to be solvable is that
\begin{equation}\label{0.6}
\kappa\ge \nu_-(S_n) ,
\end{equation}
where $\nu_-(S_n)$ is the total multiplicity of all negative
eigenvalues of $S_n$. The method we use for the solution of the
moment problem ${ MP}_\kappa({\bf s},2n)$ and ${ IP}_\kappa(s,\ell)$
for $\kappa\ge\nu_-(S_n)$ is based on the Schur-Chebyshev recursion
algorithm, studied in the nondegenerate situation by
M.~Derevjagin~\cite{Derev03} (see also~\cite{ADL07}).
 With this method every solution $\varphi$ of the
moment problem ${ MP}_\kappa({\bf s},2n)$ can be obtained via
   \begin{equation} \label{2.25i}
\varphi(\lambda)=
\frac{-s_{n_1-1}}{p_1(\lambda)+\varepsilon_1\varphi_1(\lambda)},
   \end{equation}
where $p_1(\lambda)=P_{n_1}(\lambda )$, $\varepsilon_1=\sgn
s_{n_1-1}$, and $\varphi_1$ is a solution of an "induced" moment
problem ${ MP}_{\kappa-\kappa_1}({\bf s}^{(1)},2(n-n_1))$ with
$\kappa_1=\nu_-(S_{n_1-1})$.

In the case of a nondegenerate moment problem the
condition~\eqref{0.6} is also sufficient for the problem ${
MP}_\kappa({\bf s},\ell)$ to be solvable and subsequent applications
of the formula~\eqref{2.25i} shows that $\cM_\kappa({\bf s},\ell)$
in the even case $(\ell=2n)$ is parametrized via the linear
fractional transformation
\begin{equation}\label{0.27}
    \varphi(\lambda)=-\frac{Q_{n_{N-1}}(\lambda)\tau(\lambda)+Q_{n_{N}}(\lambda)}
    {P_{n_{N-1}}(\lambda)\tau(\lambda)+P_{n_{N}}(\lambda)},
\end{equation}
with the parameter $\tau$ ranging over the class
$\mathbf{N}_{\kappa-\nu_-(S_n)}$ and satisfying the Nevanlinna
condition (E); see \cite{Dym89,Der99}. In this formula $P_j$ and
$Q_j$ are polynomials of the first and the second types introduced
in~\cite{Der99}. In the odd case a similar description of the sets
$\cM_\kappa({\bf s},\ell)$ and $\cI_\kappa({\bf s},\ell)$ is given
in Theorem~\ref{thm:5.1}, with the parameter $\tau$ ranging over the
class $\mathbf{N}_{\kappa-\nu_-(S_n),1}$ and
$\mathbf{N}_{\kappa-\nu_-(S_n)}$, respectively, and satisfying the
condition (O). It should be mentioned, that this result for
$\cI_\kappa({\bf s},\ell)$ can be derived also from the recent
paper~\cite{ADLRSh10} on boundary interpolation in generalized
Nevanlinna classes.

Now let us briefly describe some of the main results obtained in the
present paper for the degenerate indefinite truncated moment problem
in the even case. As in the definite case, for a degenerate problem
the condition \eqref{0.6} is not sufficient for the problem
${MP}_\kappa({\bf s},2n)$ to be solvable. The following theorem
gives some solvability criteria in the special case, where
$\kappa=\nu_-(S_{n})$; in fact, this result offers a natural
generalization for the results due to Curto and Fialkow \cite{CF91},
which were formulated in Theorem~\ref{een} above.

\begin{thm}\label{1}
Let $n_1<n_2<\dots<n_N$ be the sequence of all normal indices of a
degenerate matrix $S_n=(s_{i+j})_{i,j=0}^n$ and let
$\kappa=\nu_-(S_n). $ Then the following statements are equivalent:
\begin{enumerate}\def\labelenumi{\rm (\roman{enumi})}

\item the moment problem ${ MP}_\kappa({\bf s},2n)$ is solvable;

\item $\rank S_n=n_N$;

\item $S_n$ admits a Hankel extension $S_{n+1}$ such that
\[
\nu_-(S_{n+1})=\nu_-(S_n);
\]

\item the sequence ${\bf s}=\{s_j\}_{j=0}^{2n}$ is recursively
generated.
\end{enumerate}
If any of the assumptions (i)--(iv) is satisfied, then the problem
${ MP}_\kappa({\bf s},2n)$ has a unique solution
$\varphi(\lambda)=-\frac{Q_{n_N}(\lambda)}{P_{n_N}(\lambda)}$.
\end{thm}
As a consequence of Theorem~\ref{1} a rigidity result for
generalized Nevanlinna functions from~\cite{ADLRSh10} can be derived
(see also~\cite{Bol08}).

A new effect which appears in the indefinite case is that the
degenerate moment problem ${ MP}_\kappa({\bf s},2n)$ has infinitely
many solutions for $\kappa$ large enough. As will be shown below,
the problem ${ MP}_\kappa({\bf s},2n)$ with $\kappa> \nu_-(S_n)$ is
solvable if and only if
\begin{equation}\label{eq:SolvCond}
\kappa\ge \nu_-(S_n)+\nu_0(S_n).
\end{equation}
If, in addition, $\rank S_n=n_N+1$ and $\nu$ satisfies some
appropriate further conditions (see~\eqref{eq:nu} below), then the
solution set $\cM_\kappa({\bf s},2n)$ can be described by the
formula~\eqref{0.27}, where
\begin{equation}\label{A01}
\tau(\lambda)=\frac{\wh\tau(\lambda)}{\lambda^{2\nu_0}},
\end{equation}
and $\wh\tau$ is a function from the class
$\mathbf{N}_{\kappa-\nu}$, which satisfies~(E); see
Theorem~\ref{thm:3.2}.

On the other hand, if \eqref{eq:SolvCond} holds and $\rank
S_n>n_N+1$, then the solution set $\cM_\kappa({\bf s},2n)$ is
described by~\eqref{0.27}, where
\begin{equation}\label{A02}
\tau(\lambda)=\frac{-\varepsilon}{\lambda^{2\nu_0}(\wh
p(\lambda)+\varepsilon\wh\tau(\lambda))},
\end{equation}
$\wh p$ is a polynomial of degree $n-\nu_0+1$ (as given
in~\eqref{4.5b} below), and $\wh\tau$ is a function from the class
$\mathbf{N}_{\kappa-\nu}$, which satisfies~(E); see
Theorem~\ref{thm:5.6}.

In the odd case the degenerate indefinite truncated moment problem
can be treated analogously. A condition, similar
to~\eqref{eq:SolvCond} appeared in~\cite{Wor97} as a solvability
condition for a degenerate indefinite Nevanlinna-Pick interpolation
problem.

In Section 2 the basic tools needed in this paper are given.
Solutions to the so-called basic moment and interpolation problems
will be described in Section 3. Section 4 describes a general
Schur-Chebyshev recursion algorithm, which makes use of the normal
indices of the associated Hankel matrix $S_n=(s_{i+j})_{i,j=0}^n$
defined in \eqref{0.3}. In Section 5 solvability criteria and
complete descriptions for the set of solutions of the problems
${MP}_\kappa({\bf s},\ell)$ and ${IP}_\kappa({\bf s},\ell)$ in the
general setting are established. Finally, Appendix contains some
results on block matrices, which are needed in this paper; however,
they may be also of independent interest: for instance,
Lemma~\ref{lem:3.1} gives an extension of a well-know result on
nonnegative block matrices.

\section{Preliminaries}
\subsection{Canonical factorizations of generalized Nevanlinna functions}

The definition of the class $ \mathbf{N}_\kappa$ of generalized
Nevanlinna functions is given in the Introduction. Clearly, if
$\varphi \in \mathbf{N}_\kappa$ and $\varphi\not\equiv 0$, then also
$-1/\varphi \in \mathbf{N}_\kappa$.

Recall (see~\cite{KL81}), that the point $\alpha\in\dR$ is called a
\textit{generalized pole of nonpositive type} (GPNT) of the function
$\varphi\in {\bf N}_\kappa$ with multiplicity
$\kappa_\alpha(\varphi)$ if
\begin{equation}
\label{gpol} -\infty < \lim_{z\widehat{\rightarrow }\alpha}
(z-\alpha)^{2\kappa_{\alpha}+1}\varphi(z) \leq 0,\quad 0 <
\lim_{z\widehat{\rightarrow }\alpha}
(z-\alpha)^{2\kappa_{\alpha}-1}\varphi(z) \leq \infty.
\end{equation}
Similarly, the point $\infty$ is called a \textit{generalized pole
of nonpositive type} of $\varphi$ with multiplicity
$\kappa_\infty(\varphi)$ if
\begin{equation}
\label{infgpol} 0\leq\lim_{z\widehat{\rightarrow }\infty }
\frac{\varphi(z)}{z^{2\kappa_{\infty}+1}} < \infty,\quad
-\infty\leq\lim_{z\widehat{\rightarrow }\infty }
\frac{\varphi(z)}{z^{2\kappa_{\infty}-1}} < 0.
\end{equation}

A point $\beta\in\dR\cup\{\infty\}$ is called a \emph{generalized
zero of nonpositive type} (GZNT) of the function $\varphi\in {\bf
N}_\kappa$ if $\beta$ is a generalized pole of nonpositive type of
the function $-1/\varphi$. The multiplicity $\pi_\beta(\varphi)$ of
the generalized zero of nonpositive type $\beta$ of $\varphi$ can be
characterized by the inequalities:
\begin{equation}
\label{gzero} 0 < \lim_{z\widehat{\rightarrow }\beta}
\frac{\varphi(z)}{(z-\beta)^{2\pi_{\beta}+1}} \leq \infty,\quad
-\infty < \lim_{z\widehat{\rightarrow }\beta}
\frac{\varphi(z)}{(z-\beta)^{2\pi_{\beta}-1}} \leq 0.
\end{equation}
Similarly, the point $\infty$ is a \emph{generalized zero of
nonpositive type} of $\varphi$ with multiplicity
$\pi_\infty(\varphi)$ if
\begin{equation}
\label{infgzer} -\infty \leq\lim_{z\widehat{\rightarrow }\infty }
{z^{2\pi_{\infty}+1}}\varphi(z) < 0,\quad 0\leq
\lim_{z\widehat{\rightarrow }\infty }
{z^{2\pi_{\infty}-1}}\varphi(z) < \infty.
\end{equation}

\begin{rem} \label{rem:2.1}
If $\varphi_1 \in \mathbf{N}_{\kappa_1}$ and $\varphi_2 \in
\mathbf{N}_{\kappa_2}$ then $\varphi_1+\varphi_2$ belongs to
$\mathbf{N}_\kappa$, where $\kappa \le \kappa_1+\kappa_2$. It was
shown by M.G. Kre\u{\i}n and H.~Langer in~\cite{KL81} for
$\varphi\in {\bf N}_\kappa$ that the total multiplicity of poles
(zeros) in $\dC_+$ and generalized poles (zeros) of nonpositive type
in $\dR\cup\{\infty\}$ is equal to $\kappa$. As a corollary of this
result one obtains that if $\varphi_1 \in \mathbf{N}_{\kappa_1}$ and
$\varphi_2 \in \mathbf{N}_{\kappa_2}$ have no common poles in
$\dC_+$ and common generalized poles  of nonpositive type in
$\dR\cup\{\infty\}$ then
$\varphi_1+\varphi_2\in\mathbf{N}_{\kappa_1+\kappa_2}$.
\end{rem}

The generalized poles and zeros of nonpositive type of a generalized
Nevanlinna function give rise to the following factorization result
(\cite{DLLSh}, see also \cite{DHS1}).

\begin{thm}\label{thm:2.1}
Let $\varphi\in {\bf N}_\kappa$ and let $\alpha_1,\dots,\alpha_l$
($\beta_1,\dots,\beta_m$) be all the generalized poles (zeros) of
nonpositive type of $\varphi$ in $\dR$ and the poles (zeros) of
$\varphi$ in $\dC_+$ with multiplicities $\kappa_1,\dots,\kappa_l$
($\pi_1,\dots,\pi_m$). Then the function $\varphi$ admits a (unique)
canonical factorization of the form
\begin{equation}
\label{CanF1} \varphi(z)=r(z)r^\#(z)\varphi_0(z),
\end{equation}
where $\varphi_0\in {\bf N}_0$, $r^\#(z)=\overline{r(\bar{z})}$, and
$r={p}/{ q}$ with relatively prime polynomials
\[
 p(z)=\prod_{j=1}^m(z-\beta_j)^{\pi_j}, \quad
 q(z)=\prod_{j=1}^l(z-\alpha_j)^{\kappa_j},
\]
of degree $\kappa-\pi_\infty(\varphi)$ and
$\kappa-\kappa_\infty(\varphi)$, respectively.
\end{thm}

It follows from~\eqref{CanF1} that the function $\varphi$ admits the
(factorized) integral representation
\begin{equation}
\label{IntR} \varphi(z)=
r(z)r^\#(z)\left(a+bz+\int_\dR\left(\frac{1}{t-z}
-\frac{t}{1+t^2}\right)d\rho(t)\right), \quad r=\frac{p}{ q},
\end{equation}
where $a\in\dR$, $b\ge 0$, and $\rho(t)$ is a nondecreasing function
satisfying the integrability condition
\begin{equation}
\label{Nint} \int_\dR \frac{d\rho(t)}{t^2+1}<\infty.
\end{equation}

\subsection{The subclasses ${\bf N}_{\kappa,-\ell}$
 of generalized Nevanlinna functions}
 \label{sec3.1}

\begin{defn}\label{def:Nk1}(see \cite{DHS1})
A function $\varphi \in {\bf N}_{\kappa}$ is said to belong to the
subclass ${\bf N}_{\kappa,1}$, if
\begin{equation}\label{eq:N1}
    \lim_{z\widehat{\rightarrow }\infty } \frac{\varphi\left( z\right)}{
 z}=0 \mbox{ and } \int_\eta^\infty \frac{|\mbox{Im } \varphi \left(
 iy\right)| }{y} \,dy<\infty,
\end{equation}
with $\eta>0$ large enough. Similarly, a function $\varphi\in {\bf
N}_{\kappa}$ is said to belong to the subclass ${\bf N}_{\kappa,0}$,
if
\begin{equation}\label{eq:N0}
\lim_{z\widehat{\rightarrow }\infty }\frac{\varphi(z)}{ z}=0 \mbox{
and } \limsup_{z \widehat{\to} \infty} |z\, \IM \varphi (z)| <
\infty.
\end{equation}

\end{defn}
\begin{rem} \label{rem:2.6}
Every function $\varphi \in {\bf N}_{\kappa,1}$ has a nontangential
limit $\lim_{\lambda\widehat{\rightarrow }\infty }\varphi(\lambda)$
at infinity. As was shown in~\cite{DHS1} the following implication
holds:
\[
\varphi \in \mathbf{N}_{\kappa,1},
\,\,\,\lim_{\lambda\widehat{\rightarrow }\infty }\varphi(\lambda)\ne
0 \Rightarrow  -1/\varphi \in \mathbf{N}_{\kappa,1}.
\]
\end{rem}

In the following theorem the subclasses ${\bf N}_{\kappa,1}$ and
${\bf N}_{\kappa,0}$ are characterized in terms of the  integral
representation~\eqref{IntR}.
\begin{thm}
\label{N1} {\rm (\cite{DHS1})} For $\varphi \in {\bf N}_\kappa$ and
$\ell=0,1$ the following statements are equivalent:
\begin{enumerate}
\def\labelenumi{\rm (\roman{enumi})}
\item
$\varphi$ belongs to ${\bf N}_{\kappa,\ell}$;
\item
$\varphi$ has the integral representation \eqref{IntR} with $\deg
 q-\deg  p=\pi_\infty(\varphi)>0$, or with $\deg  p=\deg
q$ ($\pi_\infty(\varphi)=0$), $b=0$, and
\begin{equation}
\label{N1int} \int_{\dR} \,(1+|t|)^{-\ell} d\rho(t)<\infty.
\end{equation}
\end{enumerate}
\end{thm}
\begin{rem}
\label{N0o} If $\varphi\in{\bf N}_{\kappa,0}$ then  the statement
(ii) in Theorem~\ref{N1} can be strengthened in the sense that for
every function $\varphi\in {\bf N}_{\kappa,0}$ there are real
numbers $\gamma$ and $s_0$, such that
\begin{equation}
\label{a20} \varphi\left( z\right)=\gamma-\frac{s_0}{z}+o\left(
\frac{1}{z}\right), \quad z\widehat{\rightarrow }\infty.
\end{equation}
\end{rem}

\begin{defn}\label{def:Nkl} (\cite{DHS07})
A function $\varphi\in {\bf N}_\kappa$ is said to belong to the
subclass ${\bf N}_{\kappa,-2n}$, $n\in\dN$, if there are real
numbers $\gamma$ and $s_0,\dots,s_{2n-1}$ such that the function
\begin{equation}
\label{tQ} {\varphi}^{[2n]}(z)=z^{2n}\left(
\varphi(z)-\gamma+\sum_{j=1}^{2n}\frac{s_{j-1}}{z^j}\right)
\end{equation}
is $O(1/z)$ as $z\widehat{\rightarrow }\infty$. Moreover,
$\varphi\in {\bf N}_\kappa$ is said to belong to the subclass ${\bf
N}_{\kappa,-2n+1}$ if the function ${\varphi}^{[2n]}$ in \eqref{tQ}
belongs to ${\bf N}_{\kappa',1}$ for some $\kappa'\in\dN$,
$\kappa'\le \kappa$.
\end{defn}

As was shown in~\cite{DHS07}, the following inclusions are satisfied
\begin{equation}
\label{subincl} \cdots \subset {\bf N}_{\kappa,-2n-1}\subset {\bf
N}_{\kappa,-2n} \subset {\bf N}_{\kappa,-2n+1}\subset\cdots \subset
{\bf N}_{\kappa,0}\subset {\bf N}_{\kappa,1}.
\end{equation}

The subclasses ${\bf N}_{\kappa,-\ell}$, $\ell \in \dN$, can be
characterized by means of  the integral representation of $\varphi$
in~\eqref{IntR}.

\begin{thm}
\label{subcl}  (\cite{DHS07}) For $\varphi\in {\bf N}_\kappa$ the
following statements are equivalent:
\begin{enumerate}
\def\labelenumi{\rm (\roman{enumi})}
\item $\varphi\in {\bf N}_{\kappa,-\ell}$, $\ell \in {\dN}$;

\item $\varphi$
has an integral representation \eqref{IntR} with
$\pi_\infty(\varphi)=\deg  q-\deg  p \ge 0$ (and $b=0$ if
$\pi_\infty(\varphi)=0$), such that
\begin{equation}
\label{Nlint} \int_{\dR} \,(1+|t|)^{\ell-2\pi_\infty}
d\rho(t)<\infty.
\end{equation}
\end{enumerate}
\end{thm}
\begin{rem}\label{rem:Nkl}
By Definition~\ref{def:Nkl} every function $\varphi\in {\bf
N}_{\kappa,-\ell}$ with odd $\ell$ admits the asymptotic expansion
\begin{equation}\label{eq:asNkl}
\varphi(\lambda)=\gamma-\frac{s_0}{\lambda}-\dots-\frac{s_{\ell}}{\lambda^{\ell+1}}
+o\left(\frac{1}{\lambda^{\ell+1}}\right), \quad \lambda \wh \to
\infty.
\end{equation}
If $\varphi\in {\bf N}_{\kappa,-\ell}$ and $\ell$ is even due to
Theorem~\ref{subcl} there exists a real number $s_{2n}$, such
that~\eqref{eq:asNkl} holds. Conversely, if  $\varphi\in {\bf
N}_{\kappa}$ and satisfies~\eqref{eq:asNkl} for even $\ell$, then
$\varphi\in {\bf N}_{\kappa,-\ell}$. This proves that in the even
case ($\ell=2n$) $\cM_\kappa(s,\ell)=\cI_\kappa(s,\ell)$, while in
the odd case ($\ell=2n+1$) the
$\cM_\kappa(s,\ell)\subset\cI_\kappa(s,\ell)$.
\end{rem}
The following Lemma is immediate from Definition~\ref{def:Nkl},
Theorem~\ref{subcl} and Remark~\ref{rem:2.1}.
\begin{lem}\label{lem:2.2}
Let  $\varphi\in \mathbf{N}_{\kappa,-\ell}$ let $\nu_0\le
\min\{\pi_\infty(\varphi),\ell/2\}$  and let
\begin{equation}\label{eq:cancel}
    \wh\varphi(\lambda)=\lambda^{2\nu_0} \varphi(\lambda).
\end{equation}
Then $\wh\varphi\in \mathbf{N}_{\kappa-\nu,-(\ell-2\nu_0)}$, where
\begin{equation}\label{eq:nu}
    \nu=\left\{\begin{array}{ll}
  \nu_0, & \mbox{if }\quad\kappa_0(\varphi)> 0; \\
  \kappa_0(\varphi), & \mbox{if }\quad 0\le \kappa_0(\varphi)\le\nu_0; \\
       \end{array}\right.
\end{equation}
Conversely, if $\wh\varphi\in \mathbf{N}_{\wh\kappa,-\wh\ell} $, and
$\kappa_\infty(\wh\varphi)= 0$, $\varphi$ and $\nu$ are given
by~\eqref{eq:cancel} and~\eqref{eq:nu}, then $\varphi\in
\mathbf{N}_{\wh\kappa+\nu,-(\wh\ell+2\nu_0)}$ and $\nu_0\le
\pi_\infty(\varphi)$.
\end{lem}

\begin{proof}
1) Let  $\varphi\in \mathbf{N}_{\kappa,-\ell}$. Due to
Theorem~\ref{subcl} $\varphi$ admits the canonical
factorization~\eqref{CanF1}, where the measure $\rho$ in the
integral representation~\eqref{IntR} satisfies the
condition~\eqref{Nlint}. It follows from~\eqref{CanF1}
and~\eqref{eq:nu} that $\wh\varphi$ admits the canonical
factorization
\[
\wh\varphi=\frac{\wh p(\lambda)\wh p^\#(\lambda)}{\wh q(\lambda)\wh
q^\#(\lambda)}\varphi_0(\lambda),
\]
with the same function $\varphi_0\in \mathbf{N}_0$ and
\[
\pi_\infty(\wh\varphi)=\deg \wh q-\deg\wh
p=\pi_\infty(\varphi)-\nu_0.
\]
Hence, the condition~\eqref{Nlint} takes the form
\begin{equation}
\label{Nlin} \int_{\dR} \,(1+|t|)^{\wh\ell-2\pi_\infty(\wh\varphi)},
d\rho(t)<\infty.
\end{equation}
where $\wh\ell=\ell-2\nu_0$. Due to Theorem~\ref{subcl} the latter
condition is equivalent to the inclusion
\[
\wh\varphi\in \mathbf{N}_{\wh\kappa,-(\ell-2\nu_0)}
\]
with $\wh\kappa\le \kappa$.

2) Since $\nu_0\le \pi_\infty(\varphi)$ then neither $\varphi$ nor
$\wh\varphi$ has a GPNT at $\infty$. Assume that
$\kappa_0(\varphi)>\nu_0$. Then both $\varphi$ and $\wh\varphi$ have
GPNTs at $0$ and $\kappa_0(\varphi)=\kappa_0(\wh\varphi)+\nu_0$.
Counting the total pole multiplicities of $\wh\varphi$ one obtains
$\wh\varphi\in\mathbf{N}_{\kappa-\nu_0}$ by Remark~\ref{rem:2.1}.

Assume that $0<\kappa_0(\varphi)\le\nu_0$. Then $\wh\varphi$ has no
GPNT at $0$ and $\varphi$ has a GPNT at $0$ of multiplicity
$\kappa_0(\varphi)$. Therefore,
$\wh\varphi\in\mathbf{N}_{\kappa-\kappa_0(\varphi)}$.

And finally, if $\kappa_0(\varphi)= 0$ and $\pi_0(\varphi)\ge 0$,
then neither $\wh\varphi$ nor $\varphi$ have a GPNT at $0$ and thus
$\wh\varphi\in\mathbf{N}_{\kappa}$. All the above statements are
easily reversed.
\end{proof}

\subsection{Toeplitz matrices}
A sequence $({\bf c},n):=(c_0,\ldots, c_n)$ of (real or complex)
numbers determines an upper triangular Toeplitz matrix
$T(c_0,\ldots, c_n)$ of order $(n+1)\times(n+1)$ with entries
$t_{i,j}=c_{j-i}$ for $i\le j$ and $t_{i,j}=0$ for $i>j$.
\begin{equation}\label{TSid1}
T(c_m,\ldots, c_j)=\left(
 \begin{array}{cccccc}
         c_m &  \dots      & c_j \\
          & \ddots & \vdots \\
       &  & c_{m} \
 \end{array} \right), \quad 0\le m<j\le n.
\end{equation}
Clearly, $T(c_m,\ldots, c_j)^*= J_{j-m+1} T(\bar c_m,\ldots, \bar
c_j)J_{j-m+1},$ where
\begin{equation}\label{TSid2}
 J_{j-m+1}=
\begin{pmatrix}
  {\bf 0} &  & 1 \\
          & \adots &  \\
    1 & & {\bf 0}
\end{pmatrix}.
\end{equation}
In particular, the coefficients of the asymptotic expansions
\begin{equation}\label{AsExp}
\begin{split}
   c(\lambda)&=
c_0+\frac{c_{1}}{\lambda}+\dots+\frac{c_{n}}{\lambda^{n}}
+o\left(\frac{1}{\lambda^{n}}\right), \quad\lambda\widehat{\rightarrow}\infty;\\
d(\lambda)&=
d_0+\frac{d_{1}}{\lambda}+\dots+\frac{d_{n}}{\lambda^{n}}
+o\left(\frac{1}{\lambda^{n}}\right),
\quad\lambda\widehat{\rightarrow}\infty.
\end{split}
\end{equation}
determine the Toeplitz matrices
\begin{equation}\label{Toepl}
T(c_0,\ldots, c_n)
   =\left(%
\begin{array}{ccc}
  c_0 & \dots & c_n \\
   & \ddots & \vdots \\
   &  & c_0 \\
\end{array}%
\right), \,\,
T(d_0,\ldots, d_n)=\left(%
\begin{array}{ccc}
  d_0 & \dots & d_n \\
   & \ddots & \vdots \\
   &  & d_0 \\
\end{array}%
\right).
\end{equation}

\begin{lem}\label{lem:2.3}
Let the functions $c$ and $d$ (meromorphic on $\dC\setminus\dR$)
have the asymptotic expansions \eqref{AsExp} and let
$a(\lambda)=c(\lambda)d(\lambda)$ have the asymptotic expansion
\[
a(\lambda)=
a_0+\frac{a_{1}}{\lambda}+\dots+\frac{a_{n}}{\lambda^{n}}
+o\left(\frac{1}{\lambda^{n}}\right),
\quad\lambda\widehat{\rightarrow}\infty.
\]
 Then the first $n+1$ coefficients of the
asymptotic expansion of $a(\lambda)$ can be found by
 \begin{equation}\label{cdExp}
 T(a_0,\ldots, a_n)= T(c_0,\ldots, c_n) T(d_0,\ldots, d_n).
\end{equation}
\end{lem}

\begin{lem}\label{polylem}
The formula
\begin{equation}\label{eq:2.5a}
p(\lambda)=\frac{1}{\det S_{m}}\det \left(
\begin{array}{cccc}
               &              &  s_{m}&  s_{m+1} \\
               &  \adots      &\adots &   \vdots       \\
         s_{m} &  s_{m+1}     & \dots &  s_{2m+1} \\
             1 & \lambda & \dots & \lambda^{m+1}
\end{array}
\right),
\end{equation}
where $s_{j}=0, j<m$, $s_m\neq 0$ ($m\ge 0$), and $S_m$ is as in
\eqref{0.3}, defines a monic polynomial $p(\lambda)=\sum_{j=0}^{m+1}
p_j\lambda^j$ of degree $m+1$ whose coefficients $p_j$, $1\le j\le
m+1$, satisfy the matrix equality
\begin{equation}\label{eq:4.5b}
 T(p_{m+1},\dots, p_{1}) \, T(s_{m},\dots,s_{2m})
 =s_{m}I_{m+1}\quad (p_{m+1}=1)
\end{equation}
for an arbitrary real number $s_{2m+1}$.
\end{lem}
\begin{proof}
Evaluating the determinant in \eqref{eq:2.5a} with respect to the
last row shows immediately that $p(\lambda)$ is a monic polynomial
of degree $m+1$. To see that the coefficients $p_j$ of $p(\lambda)$
in \eqref{eq:2.5a} satisfy \eqref{eq:4.5b} substitute $\lambda^j$ by
$s_{j+k-1}$ for $j=0,1,\ldots,m+1$ ($s_{-1}=0$) in the formula
\eqref{eq:2.5a}. Then for $k=0$ the evaluation of the determinant in
\eqref{eq:2.5a} yields the equality
\[
 \Sigma_{j=1}^{m+1} p_js_{j-1}=s_m,
\]
and for $k=1,\ldots,m$ one obtains
\[
 \Sigma_{j=1}^{m+1} p_js_{j+k-1}=0.
\]
This means that the polynomial $p$ defined by \eqref{eq:2.5a}
automatically satisfies~\eqref{eq:4.5b} for arbitrary
$s_{2m+1}\in\dR$. Note that $s_{2m+1}$ only appears in the constant
coefficient $p_0$ of $p(\lambda)$, which can be seen e.g. by
evaluating the determinant in \eqref{eq:2.5a} with respect to the
last column.
\end{proof}

\subsection{Asymptotic expansions of certain fractional transforms}
In the next lemma the polynomial $p(\lambda)$ defined in
Lemma~\ref{polylem} (see \eqref{eq:2.5a}) appears when inverting an
associated asymptotic expansion.

\begin{lem}\label{cor:2.4}
Let $({\bf s},\ell)=(s_j)_{j=0}^\ell$ be a sequence of real numbers
such that $s_{j}=0$, $j<m$, and $s_m\neq 0$, $\ell\ge 2m$, and let
the monic polynomial $p(\lambda)=\sum_{j=0}^{m+1} p_j\lambda^j$ be
defined by \eqref{eq:2.5a}. Then a function $\varphi$ (meromorphic
on $\dC\setminus\dR$) admits the asymptotic expansion
\begin{equation}\label{eq:2.4}
 \varphi(\lambda)=
 -\frac{s_{m}}{\lambda^{m+1}}-\dots-\frac{s_{\ell}}{\lambda^{\ell+1}}
  +o\left(\frac{1}{\lambda^{\ell+1}}\right),
\end{equation}
if and only if the function $-s_{m}/\varphi(\lambda)$ admits the
asymptotic expansion
\begin{equation} \label{eq:2.4a}
-s_{m}/\varphi(\lambda) =  p(\lambda)+\varepsilon \tau(\lambda)\quad
\lambda \wh \to \infty,
\end{equation}
where $\varepsilon=\sgn s_m$ and $\tau(\lambda)$ satisfies one of
the following conditions:
\begin{enumerate}\def\labelenumi{(\roman{enumi})}
\item if $\ell=2m$ then $\tau(\lambda)=o(\lambda)$, $\lambda \wh \to
\infty$, and in \eqref{eq:2.5a} $s_{2m+1}$ can be an arbitrary real
number;

\item if $\ell=2m+1$ then $\tau(\lambda)=o(1)$ as $\lambda \wh \to \infty$;

\item if $\ell>2m+1$ then $\tau(\lambda)$ has the asymptotic expansion
\begin{equation} \label{eq:2.4E}
 \tau(\lambda)=
 -\frac{\wh s_{0}}{\lambda}-\dots-\frac{\wh s_{\ell-2m-2}}{\lambda^{\ell-2m-1}}
  +o\left(\frac{1}{\lambda^{\ell-2m-1}}\right), \quad \lambda \wh \to \infty,
\end{equation}
where the sequence $(\widehat{\bf s},\ell-2m-2)$ is determined by
the matrix equation
\[
     T(p_{m+1},\ldots,p_0,-\varepsilon \wh s_0,\ldots,-\varepsilon
 \wh s_{\ell-2m-2}) \, T(s_m,\dots,s_\ell) =s_m I_{\ell-m+1}.
\]
\end{enumerate}
\end{lem}
\begin{proof}
It is clear that the function $\varphi$ admits the asymptotic
expansion \eqref{eq:2.4} if and only if
$c(\lambda):=-\lambda^{m+1}\varphi(\lambda)$ admits the asymptotic
expansion of the form \eqref{AsExp} with $n=\ell-m$ and, moreover,
by standard inversion of expansions, this is equivalent for
$d(\lambda):=1/c(\lambda)$ to admit the asymptotic expansion of the
form \eqref{AsExp} with $n=\ell-m$. Now by substituting the
expansions for the terms in the formula
\begin{equation}\label{eqprod1}
 \left(\frac{p(\lambda)+\varepsilon\tau(\lambda)}{\lambda^{m+1}}\right)
  \left(-\lambda^{m+1}\varphi(\lambda)\right)=s_m
\end{equation}
and applying Lemma~\ref{lem:2.3} it is seen that
\begin{equation}\label{eqprod2}
 \frac{p(\lambda)+\varepsilon\tau(\lambda)}{\lambda^{m+1}}
 =p_{m+1}+\frac{p_m}{\lambda}+\dots + \frac{p_{2m-\ell+1}}{\lambda^{\ell-m}}
 +o\left(\frac{1}{\lambda^{\ell-m}}\right), \quad \lambda \wh \to
 \infty,
\end{equation}
where the coefficients $p_j$ are determined by the matrix equation
\begin{equation}\label{eqprod3}
 T(p_{m+1},\dots,p_{2m-\ell+1}) \, T(s_m,\dots,s_\ell)=s_m I_{\ell-m+1}.
\end{equation}
In particular, $p(\lambda)$ appearing in \eqref{eq:2.4a} is a
polynomial of degree $m+1$ whose coefficients satisfy the matrix
equality \eqref{eq:4.5b} when $\ell\ge 2m$. Hence, $p(\lambda)$ in
\eqref{eq:2.4a} can be taken to be the polynomial defined by
\eqref{eq:2.5a} in Lemma~\ref{polylem}.

(i) If $\ell=2m$ then the formula \eqref{eqprod2} shows that the
function $\tau(\lambda)$ in \eqref{eq:2.4a} satisfies
$\tau(\lambda)=o(\lambda)$ as $\lambda \wh \to \infty$, independent
from the selection of the real number $s_{2m+1}$ in \eqref{eq:2.5a}.

(ii) If $\ell=2m+1$ then the formula \eqref{eqprod2} shows that
$\tau(\lambda)=o(1)$ as $\lambda \wh \to \infty$.

(iii) If $\ell>2m+1$ then the formula \eqref{eqprod2} shows that the
function $\tau(\lambda)$ in \eqref{eq:2.4a} has the asymptotic
expansion \eqref{eq:2.4E} with $\wh s_j=-\varepsilon p_{-j-1}$,
$j=0,\dots,\ell-2m-2$, with coefficients $p_j$ as in
\eqref{eqprod3}.
\end{proof}

Associate with the expansion \eqref{eq:2.4} the $(m+1)\times(m+1)$
matrix $S_m$ as in \eqref{0.3}. Then
\begin{equation}\label{TSid3}
 S_m=
 \left(
 \begin{array}{ccccccc}
        &        & s_m \\
        & \adots & \vdots \\
    s_m & \dots & s_{2m} \
 \end{array}
 \right)=J_{m+1} T(s_m, \dots,  s_{2m}),
\end{equation}
where $T({\bf s},m,2m)$ and $J_{m+1}$ are as in \eqref{TSid1} and
\eqref{TSid2}. Let the monic polynomial $p(\lambda)$ of degree $m+1$
be defined by the formula~\eqref{eq:2.5a}, where $s_{2m+1}$ is an
arbitrary real number in the case $\ell=2m$. Then it follows from
\eqref{eq:4.5b} that
\[
T(p_{m+1},\dots,p_1)J_{m+1}=s_mS_m^{-1}
\]
and hence
\begin{equation}
\label{negsq0a}
 s_m\frac{p(\lambda)-p(\bar\omega)}{\lambda-\bar\omega}=
 \left(
 \begin{array}{ccc}
      1 & \dots   & \lambda^m   \\
  \end{array}
 \right)
 s_m^2 S_m^{-1}
 \left(
 \begin{array}{c}
     1 \\ \vdots \\  \bar\omega^m
  \end{array}
 \right),
\end{equation}
so that $s_{m} p \in \mathbf{N}_{\nu_-(S_m)}$ (as $s_m\neq 0$),
i.e., the negative index of the generalized Nevanlinna function $s_m
p(\lambda)$ is equal to $\nu_-(S_m)$, the number of negative
eigenvalues of the matrix $S_m$.

When the function  $\varphi(\lambda)$ in \eqref{eq:2.4} is a
generalized Nevanlinna function the statements in the previous lemma
can be specified further. The next result shows how the classes
$\mathbf{N}_{\kappa,-\ell}$ behave under linear fraction transforms;
for this it suffices to consider the transform $\varphi(\lambda)\to
-1/\varphi(\lambda)$; as in Lemma~\ref{cor:2.4} the result is
expressed via the function $\tau(\lambda)$ which will appear in
later sections, too.

\begin{lem}\label{lem:2.15}
Let the notations and assumptions be as in Lemma~\ref{cor:2.4}. Then
the following assertions hold:
\begin{enumerate}\def\labelenumi{(\roman{enumi})}
\item $\varphi \in \mathbf{N}_\kappa$ if and
only if $\tau \in \mathbf{N}_{\kappa-\nu_-(S_m)}$;
\item $\varphi \in \mathbf{N}_{\kappa,-2m}$ if and only if
$\tau \in \mathbf{N}_{\kappa-\nu_-(S_m)}$ and
$\tau(\lambda)=o(\lambda)$ as $\lambda \wh \to \infty$;

\item $\varphi \in \mathbf{N}_{\kappa,-2m-1}$ if and only if
$\tau\in \mathbf{N}_{\kappa-\nu_-(S_m),1}$ and $\tau(\lambda)=o(1)$
as $\lambda \wh \to \infty$;

\item $\varphi \in \mathbf{N}_{\kappa,-\ell}$ with $\ell>2m+1$ if and only if
$\tau \in \mathbf{N}_{\kappa-\nu_-(S_m),-(\ell-2m-2)}$ and
$\tau(\lambda)=o(1)$ as $\lambda \wh \to \infty$.
\end{enumerate}
\end{lem}
\begin{proof}
(i) The condition $\varphi \in \mathbf{N}_\kappa$ is equivalent to
$-1/\varphi \in \mathbf{N}_\kappa$ ($\varphi\neq 0$). Since $\ell\ge
2m$ it follows from Lemma~\ref{cor:2.4} that
$\tau(\lambda)=o(\lambda)$, $\lambda \wh \to \infty$. Hence the
definition of $\tau(\lambda)$ in \eqref{eq:2.4a} and
\cite[Theorem~4.1, Corollary~4.2]{DHS2001} imply that $-1/\varphi
\in \mathbf{N}_\kappa$ if and only if $\tau(\lambda)$ is a
generalized Nevanlinna function such that its negative index
$\kappa(\tau)$ satisfies
\[
 \kappa=\kappa(-1/\varphi)=\kappa(\varepsilon
 P)+\kappa(\tau)=\nu_-(S_m)+\kappa(\tau);
\]
see \eqref{negsq0a}. Hence, $\varphi \in \mathbf{N}_\kappa$ if and
only if $\tau \in \mathbf{N}_{\kappa-\nu_-(S_m)}$.

(ii) The statement for $\ell=2m$ is obtained now directly from part
(i) of Lemma~\ref{cor:2.4}.

(iii) \& (iv) Let $\ell=2k$ or $\ell=2k-1$ with $k>m$ and rewrite
the expansion of $\varphi(\lambda)$ in \eqref{eq:2.4} as follows:
\begin{equation}\label{eq:2.4C}
 \varphi(\lambda)=
 -\frac{s_{m}}{\lambda^{m+1}}-\dots-\frac{s_{2k-1}}{\lambda^{2k}}
  +\frac{C(\lambda)}{\lambda^{2k}},\quad C(\lambda)=o\left(1\right),
  \quad \quad \lambda \wh \to
 \infty.
\end{equation}
Then by~\eqref{subincl} $\varphi\in\mathbf{N}_{\kappa,-2k+1}$ if and
only if $C\in \mathbf{N}_{\kappa',1}$ and, similarly,
$\varphi\in\mathbf{N}_{\kappa,-2k}$ if and only if $C\in
\mathbf{N}_{\kappa',0}$ for some $\kappa'\le\kappa$. Now the
expansion in \eqref{eqprod2} can be rewritten as follows
\begin{equation}\label{eq:2.4D}
 \frac{p(\lambda)+\varepsilon\tau(\lambda)}{\lambda^{m+1}}
 =p_{m+1}+\frac{p_m}{\lambda}+\dots + \frac{p_{2(m-k+1)}}{\lambda^{2k-m-1}}
 +\frac{D(\lambda)}{\lambda^{2k-m-1}}, \quad \lambda \wh \to
 \infty,
\end{equation}
where $D(\lambda)$ satisfies
\[
 D(\lambda)=\frac{p_{m+1}}{s_m}\,C(\lambda)+O\left(\frac{1}{\lambda}\right),
 \quad \lambda \wh \to \infty;
\]
compare \cite[Lemma~4.1]{HSW}. The formula \eqref{eq:2.4D} is
equivalent to the following expansion for $\tau(\lambda)$:
\begin{equation}\label{eq:2.4F}
\tau(\lambda)=\varepsilon \left(\frac{p_{-1}}{\lambda}+\dots
+\frac{p_{2(m-k+1)}}{\lambda^{2(k-m-1)}}\right) +\frac{\varepsilon
D(\lambda)}{\lambda^{2(k-m-1)}},
\end{equation}
where $\varepsilon D(\lambda)=\frac{\varepsilon
p_{m+1}}{s_m}\,C(\lambda)+O\left(\frac{1}{\lambda}\right)$ as
$\lambda \wh \to \infty$ (and for $k=m+1$ the first term in the
righthand side of \eqref{eq:2.4F} is missing). Since here
$\varepsilon D(\lambda)\in \mathbf{N}_{\kappa'',j}$ for some
$\kappa''\in\dN$ is equivalent to $C\in \mathbf{N}_{\kappa',j}$ for
$j=0,1$ ($p_{m+1}=1$), the assertion $\tau \in
\mathbf{N}_{\kappa-\nu_-(S_m),-(\ell-2m-2)}$ for the values
$\ell=2k$ and $\ell=2k-1$ with $k>m$ follows from \eqref{eq:2.4F} by
the inclusions~\eqref{subincl}.
\end{proof}

In the special case $\kappa=0$ and $m=0$ the result in
Proposition~\ref{lem:2.15} implies \cite[Theorem~4.2]{HSW}. Results
analogous to that in Proposition~\ref{lem:2.15} in the special case
where $\ell=2k$ is even can be found from
\cite[Proposition~5.4]{DHS2001}, \cite[Theorem~5.4]{DHS07}; the even
case ($\ell=2k$) is easier, since then the expansion \eqref{eq:2.4}
for $\varphi \in \mathbf{N}_{\kappa}$ is equivalent to $\varphi \in
\mathbf{N}_{\kappa,-2m}$, see e.g.
\cite[Corollaries~3.4,~3.5]{DHS07}.

\section{Basic moment and interpolation problems}\label{sec3}

In this section solutions of the so-called  basic moment and
interpolation problems will be described. The treatment is divided
into two cases: nondegenerate and degenerate problems according to
$\det S_n\neq 0$ and $\det S_n=0$, where the matrix $S_n$ is as
defined in \eqref{0.3}. The even and the odd cases of the  problems
${ MP}_\kappa({\bf s},\ell)$ and ${ IP}_\kappa({\bf s},\ell)$ will
be treated in a parallel way. In the even case ($\ell=2n$) all the
statements will be formulated only for the problem ${
MP}_\kappa({\bf s},2n)$, since the problems ${ MP}_\kappa({\bf
s},2n)$ and ${ IP}_\kappa({\bf s},2n)$ are equivalent (see
Remark~\ref{rem:Nkl}). In what follows a sequence $({\bf
s},\ell):=\{s_{i}\}_{i=0}^\ell$ is said to be \textit{normalized} if
the first nonzero element of ${\bf s}$ has modulus $1$.

\subsection{Nondegenerate basic moment and interpolation problems}
Let $n=[\ell/2]$, so that either $\ell=2n$ or $\ell=2n+1$.
Nondegenerate problems ${ MP}_\kappa({\bf s},\ell)$ and ${
IP}_\kappa({\bf s},\ell)$ are said to be \textit{basic} if the
sequence $({\bf s},\ell)$ is normalized and $\det S_j=0$ for all
$j\le n-1$, or, equivalently, if
\begin{equation}\label{eq:3.1}
 s_0=s_1=\dots=s_{n-1}=0, \quad |s_n|=1.
\end{equation}
Thus a function $\varphi \in \mathbf{N}_{\kappa}$ is a solution to
the nondegenerate basic interpolation problem ${ IP}_\kappa({\bf
s},2n)$, if it satisfies the condition
\begin{equation}\label{2.2a}
\varphi(\lambda)=
-\frac{s_{n}}{\lambda^{n+1}}-\frac{s_{n+1}}{\lambda^{n+2}}-\dots-\frac{s_{\ell}}{\lambda^{\ell+1}}
+o\left(\frac{1}{\lambda^{\ell+1}}\right), \quad
\lambda\widehat{\rightarrow }\infty.
\end{equation}
If $IP_\kappa({\bf s},\ell)$ is nondegenerate and basic, then the
Hankel matrix $S_n$ has the form
\[
 S_n=
 \left(
 \begin{array}{ccccccc}
        &        & s_n \\
        & \adots & \vdots \\
    s_n & \dots & s_{2n} \
 \end{array}
 \right),
\]
where all the nonspecified entries are equal to 0. Define the monic
polynomial $p(\lambda)$ of degree $n+1$ by the
formula~\eqref{eq:2.5a}, where $m=n$ and $s_{2n+1}$ is an arbitrary
real number in the case of even $\ell$. Then it follows from
\eqref{negsq0a} that $s_{n} p \in \mathbf{N}_{\nu_-(S_n)}$. In the
case of even $\ell=2n$ the following result is a corollary of
general descriptions of $\cM_\kappa({\bf s},\ell)$ given
in~\cite{Dym89} and \cite{Der99}, and a short proof in this even
case $\ell=2n$ has been presented in~\cite{Derev03}. Here the result
both for even and odd $\ell$ is an immediate consequence of the
general transformation result given in Proposition~\ref{lem:2.15}.

\begin{lem}\label{derev}
Let $({\bf s},\ell)$ be a sequence of real numbers such
that~\eqref{eq:3.1} holds with $n=[\ell/2]$, let
$\nu_-:=\nu_-(S_n)$, let $p \in \mathbf{N}_{\nu_-}$ be the
polynomial of degree $n+1$ defined in \eqref{eq:2.5a}, and
$\varepsilon=s_n$. Then ${ MP}_\kappa({\bf s},\ell)$ and ${
IP}_\kappa({\bf s},\ell)$ are solvable if and only if
\begin{equation}\label{kap}
 \kappa \ge \nu_-.
\end{equation}
If $\kappa \ge \nu_-$ then the formula
\begin{equation}\label{4.12}
\varphi(\lambda)=-\frac{\varepsilon}
    { p(\lambda)+\varepsilon \tau(\lambda)},
\end{equation}
describes the sets $\cM_\kappa({\bf s},\ell)$ and $\cI_\kappa({\bf
s},\ell)$ as follows: in the even case
\[
\varphi\in\cM_\kappa({\bf s},\ell)\Leftrightarrow \tau \in
\mathbf{N}_{\kappa-\nu_-}\mbox{ and satisfies   (E); }
\]
and in the odd case
\[
\varphi\in\cM_\kappa({\bf s},\ell)\Leftrightarrow \tau \in
\mathbf{N}_{\kappa-\nu_-,1}\mbox{ and satisfies   (O); }
\]
\[
\varphi\in\cI_\kappa({\bf s},\ell)\Leftrightarrow \tau \in
\mathbf{N}_{\kappa-\nu_-}\mbox{ and satisfies   (O).}
\]
\end{lem}
\begin{proof}
Let $\varphi \in \cI_\kappa({\bf s},\ell)$. Then $\varphi \in
\mathbf{N}_{\kappa}$ with the expansion \eqref{2.2a}. Since
$|s_n|=1$ Proposition~\ref{lem:2.15} shows that now equivalently
\begin{equation} \label{2.4a}
 -1/\varphi(\lambda) = \varepsilon p(\lambda)+\tau(\lambda),
\end{equation}
where $\tau\in\mathbf{N}_{\kappa-\nu_-}$. In parti\-cu\-lar, the
solvability criterion \eqref{kap} and the assertions in the even
case $\ell=2n$ and the odd case $\ell=2n+1$ are obtained from and
parts (i) and (ii) of Lemma~\ref{cor:2.4}, respectively, and
Lemma~\ref{lem:2.15} (i).

The last statement for the moment problem ${ MP}_\kappa({\bf
s},\ell)$ in the odd case is implied by Lemma~\ref{lem:2.15} (iii).
\end{proof}

\subsection{Degenerate basic problems.}
Let $\ell\in\dN$ and  $n=[\ell/2]$. Degenerate moment and
interpolation problems ${ MP}_\kappa({\bf s},\ell)$ and ${
IP}_\kappa({\bf s},\ell)$ are said to be \textit{basic} if $\det
S_j=0$ for all $j\le n$ and the sequence $({\bf s}, \ell)$ is
normalized. Consequently, the set of degenerate basic moment
problems can be divided into two cases as follows:
\begin{enumerate}
\item[(A)]
$s_j=0$ for all $j=0,\dots,2n$.

\item[(B)] There is at least one nonzero moment $s_j$ for some $j=n+1,\ldots,2n$.
Let $m$ be the minimal number for which $|s_{m}|=1$  $(n<m \le 2n)$.
\end{enumerate}

\subsubsection{Degenerate basic problems: Case (A)}
In this case $s_0=\dots=s_{2n}=0$ and $\nu_0(S_n)=n+1$. Let us
denote $\nu_0:=\nu_0(S_n)$. In the next theorem descriptions of the
sets of solutions to the degenerate basic  problems ${
MP}_\kappa({\bf s},\ell)$ and ${ IP}_\kappa({\bf s},\ell)$ are
given.

\begin{lem} \label{prop:3.2}
Let $({\bf s},\ell)$ be a sequence of real numbers satisfying the
assumption (A) with $n=[\ell/2]$. Then in the even case the problems
${ MP}_\kappa({\bf s}, \ell)$ and ${ IP}_\kappa({\bf s},\ell)$ are
solvable if and only if
\begin{equation}\label{eq:3.6}
   \mbox{ either }\kappa=0,  \mbox{ or }\kappa\ge\nu_0:=\nu_0(S_n).
\end{equation}
In the odd case ${ MP}_\kappa({\bf s}, \ell)$ and ${ IP}_\kappa({\bf
s},\ell)$ are solvable if and only if
\begin{equation}\label{eq:3.7}
   \mbox{ either }\kappa=0 \mbox{ and } s_\ell=0, \mbox{ or }\kappa\ge\nu_0.
\end{equation}
If $\kappa=0$  (and  $s_\ell=0$ in the odd case), then it has the
unique solution $\varphi(\lambda)\equiv 0$.\\

If $\kappa\ge \nu_0$ and $\nu$ is given by~\eqref{eq:nu}, then  the
formula
  \begin{equation}\label{eq:3.5}
\varphi(\lambda)=\frac{\wh\varphi(\lambda)}{\lambda^{2\nu_0}},
\end{equation}
describes the sets $\cM_\kappa({\bf s},\ell)$ and $\cI_\kappa({\bf
s},\ell)$ as follows: in the even case
\[
\varphi\in\cM_\kappa({\bf s},\ell)\Leftrightarrow \wh\varphi\in
\mathbf{N}_{\kappa-\nu}\mbox{ and satisfies   (E); }
\]
and in the odd case
\[
\varphi\in\cM_\kappa({\bf s},\ell)\Leftrightarrow
\wh\varphi+s_\ell\in \mathbf{N}_{\kappa-\nu_-,1}\mbox{ and satisfies
(O) };
\]
\[
\varphi\in\cI_\kappa({\bf s,\ell})\Leftrightarrow
\wh\varphi+s_\ell\in \mathbf{N}_{\kappa-\nu_-}\mbox{ and satisfies
(O) }.
\]
\end{lem}

\begin{proof}
The case $\kappa=0$ is trivial, since if $\varphi\in\mathbf{N}_0$,
then $s_0=0$ implies $\varphi=0$.

Let $\kappa>0$ and let $\varphi\in \mathbf{N}_{\kappa}$ be a
solution of the interpolation problem $IP_\kappa({\bf s}, \ell)$.
Then it follows from~\eqref{2.2a} and (A) that in the even case
($\ell=2n$)
\begin{equation}\label{2.2e}
\varphi(\lambda)= o\left(\frac{1}{\lambda^{2n+1}}\right), \quad
\lambda\widehat{\rightarrow }\infty.
\end{equation}
and in the odd case
\begin{equation}\label{2.2o}
\varphi(\lambda)= -\frac{s_{2n+1}}{\lambda^{2n+2}}
+o\left(\frac{1}{\lambda^{2n+2}}\right), \quad
\lambda\widehat{\rightarrow }\infty.
\end{equation}
In both cases
\[
\lim_{\lambda\hat\to\infty}\lambda^{2n+1}\varphi(\lambda)=0,\quad
\]
and, hence, the multiplicity $\pi_\infty(\varphi)$ is at least
$n+1$. In view of Remark~\ref{rem:2.1} one has
\[
  \kappa\ge \pi_\infty(\varphi)\ge n+1=\nu_0.
\]
This proves the necessity of the condition $\kappa\ge \nu_0$. Due to
Lemma~\ref{lem:2.2} $\varphi$ admits the
representation~\eqref{eq:3.5} with $\wh\varphi\in
\mathbf{N}_{\kappa-\nu}$. In addition, it follows from~\eqref{2.2e}
and~\eqref{2.2o} that $\wh\varphi$ satisfy the assumption~(E) if
$\ell$ is even, and $\wh\varphi+s_\ell$ satisfy~(O) if $\ell$ is
odd.

And finally, let $\ell=2n+1$ and $\kappa\ge \nu_0=n+1$. Then by
Lemma~\ref{lem:2.2} $\varphi\in\mathbf{N}_{\kappa,-\ell}$ if and
only if $\wh\varphi\in\mathbf{N}_{\kappa-\nu,1}$. This completes the
proof.
\end{proof}

\subsubsection{Degenerate basic  problems: Case (B)} In this case
a function $\varphi$ from $ \mathbf{N}_{\kappa,\ell}$ $(
\mathbf{N}_{\kappa,-\ell})$ is a solution to the degenerate basic
problem ${ MP}_\kappa({\bf s}, \ell)$  (${ IP}_\kappa({\bf
s},\ell)$), if
\begin{equation}\label{4.2m}
     \varphi(\lambda)=
-\frac{s_{m}}{\lambda^{m+1}}-\dots-\frac{s_{\ell}}{\lambda^{\ell+1}}
+o\left(\frac{1}{\lambda^{\ell+1}}\right), \quad
\lambda\widehat{\rightarrow }\infty,
\end{equation}
with $|s_{m}|=1$ and $n<m\le 2n$, where $n=[\ell/2]$.
\begin{lem} \label{prop:3.3}
Let $({\bf s},\ell)$ be a sequence of real numbers satisfying the
assumption (B) with $n=[\ell/2]$.  Then the problems ${
MP}_\kappa({\bf s}, \ell)$ and ${ IP}_\kappa({\bf s},\ell)$ are
solvable if and only if
\begin{equation}\label{SolvCond}
    \kappa\ge k:=\nu_0+\nu_-,\quad \nu_0:=\nu_0(S_n),\,\nu_-:=\nu_-(S_n).
\end{equation}
Let the sequence $(\widehat{{\bf s}},\ell-2\nu_0)=\{\wh
s_i\}_{i=0}^{\ell-2\nu_0}$ be given by the equalities
\begin{equation}\label{eq:4.3}
\wh s_j= s_{j+2\nu_0},\quad(j=0,\dots,\ell-2\nu_0),
\end{equation}
and let $\nu$ be defined by~\eqref{eq:nu}. Then the
formula~\eqref{eq:3.5} establishes a one-to-one correspondence
between solutions $\varphi$ of the problem $IP_\kappa({\bf s},
\ell)$ and solutions $\wh\varphi$ of the nondegenerate basic problem
$IP_{\kappa-\nu}(\wh{\bf s}, \ell-2\nu_0)$. A similar statement
concerning the problems $MP_\kappa({\bf s}, \ell)$ and
$MP_{\kappa-\nu}(\wh{\bf s}, \ell-2\nu_0)$) is also true.
\end{lem}
\begin{proof}
Let $\varphi\in\cI_\kappa({\bf s},\ell)$, so that $\varphi$ belongs
to ${\bf N}_{\kappa}$ and satisfies \eqref{4.2m}. Now the matrix
$S_n$ takes the form
\[
 S_n=
 \left(
 \begin{array}{ll}
{\bf 0}_{(m-n)\times{(m-n)}}     & {\bf 0}_{{(m-n)}\times(2n -m+1)} \\
{\bf 0}_{(2n -m+1)\times{(m-n)}} & S_{[m-n,n]}    \\
 \end{array}
 \right),
 \]
where
\[
 S_{[m-n,n]}=(s_{i+j})_{i,j=m-n}^n=\left(
 \begin{array}{cccccc}
         &        & s_m \\
          & \adots & \vdots \\
      s_m & \dots & s_{2n} \
 \end{array} \right)
\]
is invertible since $s_m\neq 0$ ($|s_m|=1$). It is clear that
\[
\nu_0(S_n)=m-n>0.
\]

To determine the index $\nu_-(S_n)$ consider the following three
subcases:
\begin{enumerate}
\item[(B1)] $m$ is even and $s_{m}>0$ (denote $m=2k$);

\item[(B2)] $m$ is odd (denote $m=2k-1$);

\item[(B3)] $m$ is even and $s_m<0$ (denote $m=2k-2$).
\end{enumerate}
Then one can easily check that
\[
\nu_-(S_n)=\left\{\begin{array}{ll}
                       n-k, &  \mbox{in case (B1)}; \\
                      n-k+1, & \mbox{in case (B2)}; \\
                      n-k+2, & \mbox{in case (B3)}, \\
                    \end{array}\right.
\]
so that in each of the cases (B1)--(B3) one has
\begin{equation}\label{nu-0}
   \nu_0(S_n)+\nu_-(S_n)=k > 0.
\end{equation}
It follows from~\eqref{4.2m} that  in the case (B1)
\[
\lim_{\lambda\hat\to\infty}\lambda^{2k+1}\varphi(\lambda)=-s_{2k}<0,\quad
\lim_{\lambda\hat\to\infty}\lambda^{2k-1}\varphi(\lambda)= 0,
\]
In the cases (B2) and (B3), respectively, one obtains
\[
\begin{split}
\lim_{\lambda\hat\to\infty}\lambda^{2k+1}\varphi(\lambda)&=\infty,\quad
\lim_{\lambda\hat\to\infty}\lambda^{2k-1}\varphi(\lambda)= 0;\\
\lim_{\lambda\hat\to\infty}\lambda^{2k+1}\varphi(\lambda)&=\infty,\quad
\lim_{\lambda\hat\to\infty}\lambda^{2k-1}\varphi(\lambda)=-s_{2k-2}>0.
\end{split}
\]
Hence, in each of these cases, $\infty$ is a GZNT  of
$\varphi(\lambda)$ of multiplicity $\pi_\infty(\varphi)=k$; see
\eqref{infgzer}. By Theorem~\ref{thm:2.1} (or Remark~\ref{rem:2.1})
this implies the inequality~\eqref{SolvCond}.

Due to Lemma~\ref{lem:2.2} $\varphi$ admits the
representation~\eqref{eq:3.5} where
$\wh\varphi\in\mathbf{N}_{\kappa-\nu}$. Clearly, the
expansion~\eqref{4.2m} can be rewritten as
\[
     \wh\varphi(\lambda)=
-\frac{s_{m}}{\lambda^{m-2\nu_0+1}}-\dots-\frac{s_{\ell}}{\lambda^{\ell-2\nu_0+1}}
+o\left(\frac{1}{\lambda^{\ell-2\nu_0+1}}\right), \quad
\lambda\widehat{\rightarrow }\infty,
\]
Therefore, $\wh\varphi\in\cI_{\kappa-\nu}(\wh{\bf s}, \ell-2\nu_0)$.
These arguments can be reversed to obtain the converse statement.

In the case where $\ell$ is odd it follows from Lemma~\ref{lem:2.2}
that $\varphi\in\mathbf{N}_{\kappa,-\ell}$ if and only if
$\wh\varphi\in\mathbf{N}_{\kappa-\nu,-(\ell-2\nu_0)}$. This proves
the statement concerning the set $\cM_\kappa({\bf s}, \ell)$.
\end{proof}

The following Lemma summarizes the results of Lemmas~\ref{derev} and
\ref{prop:3.3}.

\begin{lem}\label{prop:4.1}
Let $({\bf s},\ell)$ be a sequence of real numbers satisfying the
assumption (B) with $n=[\ell/2]$, let~\eqref{SolvCond} hold, let
$\nu$ be defined by~\eqref{eq:nu},  $\varepsilon=s_m$, and let
\begin{equation}\label{4.5b}
\wh p(\lambda)=\frac{1}{\det S_{[m-n,n]}}\det \left(
\begin{array}{cccc}
        &          &  s_{m} & s_{m+1} \\
        &  \adots  & \adots &   \vdots       \\
  s_{m} &  s_{m+1} & \dots  &  s_{2n+1} \\
  1     & \lambda  & \dots  & \lambda^{2n+1}
\end{array}
\right),
\end{equation}
where $s_{2n+1}$ is an arbitrary real number if $\ell$ is even. Then
the formula
\begin{equation}\label{4.12b1}
\varphi(\lambda)=-\frac{\varepsilon}
    {\lambda^{2\nu_0}(\wh p(\lambda)+\varepsilon\tau(\lambda))},
\end{equation}
describes the sets $\cM_\kappa({\bf s},\ell)$ and $\cI_\kappa({\bf
s},\ell)$ as follows: in the even case
\[
\varphi\in\cM_\kappa({\bf s},\ell)\Leftrightarrow \wh\varphi\in
\mathbf{N}_{\kappa-\nu_--\nu}\mbox{ and satisfies   (E); }
\]
and in the odd case
\[
\varphi\in\cM_\kappa({\bf s},\ell)\Leftrightarrow
\wh\varphi+s_\ell\in \mathbf{N}_{\kappa-\nu_--\nu,1}\mbox{ and
satisfies (O) };
\]
\[
\varphi\in\cI_\kappa({\bf s,\ell})\Leftrightarrow
\wh\varphi+s_\ell\in \mathbf{N}_{\kappa-\nu_--\nu}\mbox{ and
satisfies (O) }.
\]
\end{lem}
\begin{proof}
Assume that \eqref{SolvCond} holds. In Lemma~\ref{prop:3.3} the
problem $IP_\kappa({\bf s}, \ell)$ was reduced to the  problem
$IP_\kappa(\wh{\bf s}, \ell-2\nu_0)$. By Lemma \ref{derev} the set
$\cI_{\kappa-\nu}(\wh{\bf s}, \ell-2\nu_0)$ can be described by the
formula
\begin{equation}\label{4.3}
\wh\varphi(\lambda)=-\frac{  \varepsilon}
    {\wh p(\lambda)+\varepsilon\tau(\lambda)}
\end{equation}
where $\tau$ is a function from the class
$\mathbf{N}_{\kappa-(\nu+\nu_-)}$ such that the appropriate
condition (E) or (O) is satisfied.  Substitution of~\eqref{4.3}
into~\eqref{eq:3.5} yields~\eqref{4.12b1}.

Due to Lemma~\ref{prop:3.3} and Lemma \ref{derev} $\varphi$ belongs
to $\cM_\kappa({\bf s},\ell)$ with odd $\ell $, if and only if the
function $\tau$ belongs to $\mathbf{N}_{\kappa-(\nu+\nu_-)}$ and
satisfies the condition (O).
\end{proof}

\section{Schur algorithm}

The present approach to the degenerate moment problem is based on
the following reduction algorithm which for the nondegenerate case
with even index $\ell$ was considered in~\cite{Derev03}.

\subsection{One step reduction for moment problems which are not basic}

Let ${\bf s}=\{s_j\}_{j=0}^{\ell}$ be an arbitrary normalized
sequence of real numbers and let $S_n=(s_{i+j})_{i,j=0}^n$ be the
Hankel matrix as defined in \eqref{0.3}. Assume that $S_n\neq 0$ and
consider a sequence of \textit{normal indices} of $S_n$,
\begin{equation} \label{4.01}
 0<n_1<\dots<n_N \le n+1
\end{equation}
which are characterized by the conditions
   \begin{equation} \label{18}
   \mbox{det } S_{n_j-1} \ne 0\quad (j=1,\dots,N).
   \end{equation}
In particular, the first normal index $n_1$ is the minimal natural
number such that $\det S_{n_1-1}\ne 0$,  or , equivalently, that
\begin{equation}\label{restr}
 s_0=s_1=\dots=s_{n_1-2}=0, \quad s_{n_1-1} \ne 0.
\end{equation}
Note that the first normal index satisfies $n_1(=n_N)=n+1$ precisely
when the moment problem is nondegenerate and basic and that there
are no normal indices for moment problems which are degenerate and
basic; see Section~\ref{sec3}.

In this section it is assumed that the moment problem is not basic,
i.e., one has $n_1\le n$. Let the sequence $({\bf
s},\ell)=\{s_j\}_{j=0}^{\ell}$ be normalized and denote
\[
 \varepsilon_1=\sgn s_{n_1-1}=\pm 1.
\]
In this case a function $\varphi \in \mathbf{N}_{\kappa,-\ell}$ is a
solution to the moment problem $\cM_\kappa({\bf s},\ell)$  if
\begin{equation}\label{2.2b}
\varphi(\lambda)=
-\frac{s_{n_1-1}}{\lambda^{n_1}}-\frac{s_{n_1}}{\lambda^{n_1+1}}-\dots-\frac{s_{\ell}}{\lambda^{\ell+1}}
+o\left(\frac{1}{\lambda^{\ell+1}}\right), \quad
\lambda\widehat{\rightarrow }\infty.
\end{equation}
Then $-1/\varphi \in \mathbf{N}_\kappa$ and, moreover, by part (iii)
of Lemma~\ref{cor:2.4} $-1/\varphi $ admits the representation
   \begin{equation} \label{2.4}
-1/\varphi(\lambda)
=\varepsilon_1p_1(\lambda)+a_1^2\varphi_1(\lambda),
   \end{equation}
where $p_1(\lambda)=p_{n_1}^{(1)}\lambda^{n_1}+\dots+ p_0^{(1)}$ is
a monic polynomial of degree $n_1$ ($p_{n_1}^{(1)}=1$), defined by
the equation~\eqref{eq:2.5a} with $m=n_1-1$, and $a_1(>0)$ is chosen
in such a way that the sequence $({\bf
s}^{(1)},\ell-2n_1)=(s_i^{(1)})_{i=0}^{\ell-2n_1}$ defined by the
expansion of $\varphi_1(\lambda)$
\begin{equation}\label{2.6}
\varphi_1(\lambda)=
-\frac{s_{0}^{(1)}}{\lambda}-\frac{s_{1}^{(1)}}{\lambda^{2}}-\dots
-\frac{s_{\ell-2n_1}^{(1)}}{\lambda^{\ell-2n_1+1}}
+o\left(\frac{1}{\lambda^{\ell-2n_1+1}}\right)
\quad(\lambda\widehat{\rightarrow }\infty),
\end{equation}
is normalized. Moreover, by Proposition~\ref{lem:2.15}~(iii)
$\varphi_1$ is a generalized Nevanlinna function from the class
\[
 \mathbf{N}_{\kappa-\kappa_1,-(\ell-2n_1)}, \quad
 \kappa_1:=\nu_-(S_{n_1-1}).
\]
As was shown in Lemma~\ref{cor:2.4} the moment sequence $({\bf
s}^{(1)},\ell-2n_1)$ is uniquely defined by the matrix equations
   \begin{equation} \label{2.7}
T(s_{n_1-1},\dots,  s_{j+2n_1})T(p_{n_1}^{(1)},\dots,
p_0^{(1)},-\varepsilon_1a_1^2{s}_{0}^{(1)},\dots,
-\varepsilon_1a_1^2{s}_{j}^{(1)} ) =\varepsilon_{1}I_{j+n_1+2} ,
  \end{equation} where $0\le j\le \ell-2n_1$.

The above considerations yield the following result.

\begin{prop}\label{prop:4.2}
Let $S_{n}$ be a Hankel matrix, let $n_1$ be the first normal index
of $S_{n}$, $n_1\le n$, let the monic polynomial
$p_1(\lambda)=p_{n_1}^{(1)}\lambda^{n_1}+\dots+ p_0^{(1)}$ and the
induced  moment sequence $({\bf s}^{(1)},\ell-2n_1)$ be defined
by~\eqref{2.7}, $\varepsilon_1={s_{n_1-1}}$,
$\kappa_1:=\nu_-(S_{n_1-1})$. Then the formula
   \begin{equation} \label{2.25}
\varphi(\lambda)=T_1[\varphi_1(\lambda)]:=
\frac{-\varepsilon_1}{p_1(\lambda)+\varepsilon_1a_1^2\varphi_1(\lambda)}
   \end{equation}
establishes a one-to-one correspondence between the sets
$\cM_\kappa({\bf s},\ell)$ and $\cM_{\kappa-\kappa_1}({\bf
s}^{(1)},\ell-2n_1)$ as well as between the sets $\cI_\kappa({\bf
s},\ell)$ and $\cI_{\kappa-\kappa_1}({\bf s}^{(1)},\ell-2n_1)$.
\end{prop}

The normal indices of the induced Hankel matrix $
S_{n-n_1}^{(1)}=(s_{i+j}^{(1)})_{i,j=0}^{n-n_1}$ can be derived from
the normal indices of the original Hankel matrix $S_n$. This is
given in the next Proposition.
\begin{prop}\label{prop:4.3}
Let $n_1<n_2<\dots<n_N(\le N+1)$ be all normal indices of the Hankel
matrix $ S_{n}$.  Then the normal indices of the induced Hankel
matrix $ S_{n-n_1}^{(1)}$ are
\[
  n_2-n_1 <\dots<n_N-n_1 .
\]
\end{prop}
\begin{proof}
It follows from \eqref{2.9b} in Lemma~\ref{lem:2.9} that $ \det
S_{i-n_1}^{(1)}\ne 0$ if and only if $i=n_2,\dots,n_N$ $(n_1\le i\le
n)$.
\end{proof}
Now applying Proposition~\ref{prop:4.2} to the matrix
$S_{n-n_1}^{(1)}$ one construct a polynomial $p_2$ and a Hankel
matrix $S_{n-n_2}^{(2)}$. After $N $ inductive steps one obtains the
Hankel matrix $S_{n-n_N}^{(N)}$ of induced moments and subsequent
application of Lemma~\ref{lem:2.9} yields the following
\begin{cor} \label{cor:2.3}

       Let $S_{n-n_j}^{(j)}$ be the Hankel matrix of induced
moments after $j$ steps $(1\le j\le N)$. Then the set of normal
indices of the Hankel matrix $ S_{n-n_j}^{(j)}$ takes the form $
\{n_k-n_j\}_{k=j+1}^N$.  Moreover, for all $i$ such that $n_j\le
i\le n$ one has
\begin{equation} \label{2.18a}
\nu_\pm(S_{i-n_j}^{(j)})=\nu_\pm(S_{i})-\nu_\pm(S_{n_j-1});
 \end{equation}
 \begin{equation} \label{2.18b}
\nu_0(S_{i-n_j}^{(j)})=\nu_0(S_{i}).
  \end{equation}
In particular, the matrix $S_{n-n_N}^{(N)}$ has no normal indices
anymore, that is $\mbox{det }S_{i-n_N}^{(N)}=0$ for all $i$ such
that $n_N\le i\le n$.
\end{cor}
   \begin{proof}
The first statement is implied by the formula~\eqref{2.9b} in
Lemma~\ref{lem:2.9}. The formula \eqref{2.18a} can be obtained by
induction. Indeed, for $j=1$ the statement is contained
in~\eqref{2.9a} of Lemma~\ref{lem:2.9}. Assume that \eqref{2.18a}
holds for some $j$ $(1\le j\le N)$ and all $i$ such that $n_j\le
i\le n$. Then it follows from~\eqref{2.9a} that
\[
\nu_\pm(S_{i-n_{j+1}}^{(j+1)})
=\nu_\pm(S_{i-n_j}^{(j)})-\nu_\pm(S_{n_{j+1}-n_j-1}^{(j)}).
\]
In view of the induction assumption this yields
\[
\begin{split}
\nu_\pm(S_{i-n_{j+1}}^{(j+1)})
&=\nu_\pm(S_{i})-\nu_\pm(S_{n_j-1})-(\nu_\pm(S_{n_{j+1}-1})-\nu_\pm(S_{n_j-1}))\\
&=\nu_\pm(S_{i})-\nu_\pm(S_{n_{j+1}-1}).
\end{split}
\]
The formula~\eqref{2.18b} is immediate from \eqref{2.9b} in
Lemma~\ref{lem:2.9}.
 \end{proof}

\subsection{Algorithm}

Let us define a sequence $\kappa_1\le\dots\le\kappa_N$  by the
equalities
\begin{equation}\label{Indices}
    \kappa_j=\nu_-(S_{n_{j}-1}),\quad j=1,\dots,N.
\end{equation}
Due to Proposition~\ref{prop:4.2} on each step one obtains a linear
fractional transformation
   \begin{equation} \label{2.25a}
\varphi_{j-1}(\lambda)=\cT_j[\varphi_{j}(\lambda)]:=
\frac{-\varepsilon_j}{p_j(\lambda)+\varepsilon_ja_j^2\varphi_{j}(\lambda)},\quad
   \end{equation}
where
\[
\varepsilon_j=\mbox{sign } s^{(j-1)}_{n_j-n_{j-1}-1}(=\pm 1),\quad
a_j>0\quad (0\le j\le N-1).
\]
The transformation $\cT_j$ establishes a one-to-one correspondence
between the sets $\varphi_{j-1}\in\cM_{\kappa-\kappa_{j-1}}({\bf
s}^{(j-1)},\ell-2n_{j-1})$ and
$\varphi_{j}\in\cM_{\kappa-\kappa_{j}}({\bf s}^{(j)},\ell-2n_{j})$.
Let $W_j(\lambda)$ be the matrix
\begin{equation}\label{Wj}
W_j(\lambda)=\begin{pmatrix}0 & -\frac{\varepsilon_j}{a_j}\\
                            \varepsilon_ja_j &  \frac{p_j(\lambda)}{a_j}
                            \end{pmatrix},\quad j\in\dN.
\end{equation}
associated with the transformation $\cT_j$ $(1\le j\le N)$.

After the $j$-th step we obtain the following representation for the
solution $\varphi$ of the moment problem ${MP}_\kappa({\bf s},\ell)$
\begin{equation}\label{ContF}
\begin{split}
\varphi(\lambda)&= \cT_{1}\circ\dots\circ
\cT_{j}[\varphi_{j}(\lambda)]\\
& =-\frac{\varepsilon_1}{p_1(\lambda)}
\begin{array}{l} \\ -\end{array}
\frac{\varepsilon_1\varepsilon_2a_1^2}{p_2(\lambda)}
\begin{array}{ccc} \\ - & \cdots & -\end{array}
\frac{\varepsilon_{j-1}\varepsilon_{j}a_{j-1}^2}{p_j(\lambda)+\varepsilon_{j}a_{j}^2\varphi_{j}(\lambda)},
\end{split}
\end{equation}
where the last formula stands for the continuous fraction expansion
(this shorthand notation is often used in the literature).  The
resulting matrix $W_{[1,j]}(\lambda)$ of the linear fractional
transformation in \eqref{ContF} coincides with the product of the
matrices $W_i$ $(1\le i\le j)$
\begin{equation}\label{eq:W}
    W_{[1,j]}(\lambda)=W_1(\lambda)\dots W_j(\lambda)\quad (j\le N).
\end{equation}

   \begin{thm} \label{thm:2.2}
Let $n_1<\dots<n_N(\le n)$ be a sequence of all normal indices of
$S_n$ and let the matrix $W_{[1,j]}(\lambda)=\begin{pmatrix}
  w_{11}^{(j)}(\lambda) & w_{12}^{(j)}(\lambda) \\
  w_{21}^{(j)}(\lambda) & w_{22}^{(j)}(\lambda) \\
\end{pmatrix}$ be given by
\eqref{eq:W}. Then for every $j\le N-1$ the formula
\begin{equation}\label{2.27}
    \varphi(\lambda)=\cT_{W_{[1,j]}(\lambda)}[\varphi_{j}(\lambda)]:=\frac{w_{11}^{(j)}(\lambda)\varphi_j(\lambda)+w_{12}^{(j)}(\lambda)}
    {w_{21}^{(j)}(\lambda)\varphi_j(\lambda)+w_{22}^{(j)}(\lambda)},
\end{equation}
establishes a one-to-one correspondence between the sets
${\cI}_\kappa({\bf s},\ell)$ and  ${\cI}_{\kappa-\kappa_{j}}({\bf
s}^{(j)},\ell-2n_j)$, where ${\bf s}^{(j)}$ is defined recursively
by~\eqref{2.7} and $\kappa_{j}=\nu_-(S_{n_j-1})$. Moreover,
\[
\varphi\in{\cM}_\kappa({\bf
s},\ell)\Leftrightarrow\varphi_j\in{\cM}_{\kappa-\kappa_{j}}({\bf
s}^{(j)},\ell-2n_j)
\]

In the case, when $\det S_n=0$ the statement remains valid for
$j=N$.
   \end{thm}
   \begin{proof}  The proof is obtained by successive application of
the Schur algorithm described above and Propositions~\ref{prop:4.2},
\ref{prop:4.3}, and   Corollary~\ref{cor:2.3} to the problem
${MP}_\kappa({\bf   s},\ell)$. In the nondegenerate case this
process terminates when $j=N-1$, since $n<n_N$. In the degenerate
case Propositions~\ref{prop:4.2} can be applied one more time, since
$n\ge n_N$.
   \end{proof}
To find an explicit form of the matrix $W_{[1,j]}(\lambda)$ let us
define the so-called polynomials $P_j(\lambda)$ and $Q_j(\lambda)$
of the first and the second kind, respectively, as solutions of the
difference equation
\begin{equation}\label{DifEq}
\varepsilon_{j-1}\varepsilon_j
a_{j-1}u_{j-2}-p_j(\lambda)u_{j-1}+a_{j}u_{j}=0\,\,\,(j=\overline{1,N}),
\end{equation}
with the initial conditions
\begin{equation}\label{Init}
\begin{split}
    P_0(\lambda)&=1,\quad P_1(\lambda)=\frac{p_1(\lambda)}{a_1},\\
     Q_0(\lambda)&=0,\quad Q_1(\lambda)=\frac{\varepsilon_1}{a_1}.
     \end{split}
\end{equation}
As is easily seen from~\eqref{DifEq}
\[
\deg P_j=\sum_{i=1}^jn_i,\,\,\deg Q_j=\sum_{i=1}^{j-1}n_i\quad (j\ge
1).
\]
   \begin{thm} \label{thm:4.6}
The resolvent  matrix $W_{[1,j]}(\lambda)$ in \eqref{eq:W} admits
the following representation
\begin{equation}\label{2.26}
W_{[1,j]}(\lambda)=\begin{pmatrix}
  -\varepsilon_{j}a_{j} Q_{{j-1}}(\lambda) & -Q_{{j}}(\lambda) \\
  \varepsilon_{j}a_{j} P_{{j-1}}(\lambda) & P_{j}(\lambda) \\
\end{pmatrix},
\end{equation}
where $P_{j}$ and $Q_{j}$ $(1\le j\le N)$ are polynomials of the
first and the second kind associated with the matrix $S_n$
via~\eqref{DifEq}, \eqref{Init}.
   \end{thm}
   \begin{proof} For $j=1$ the formula~\eqref{2.26} coincides
   with~\eqref{Wj}. Proceed by induction and assume that~\eqref{2.26} holds for
   $j:=j-1$.
Then it follows from~\eqref{2.26},~\eqref{Wj} and  the difference
equation~\eqref{DifEq} that
\[
\begin{split}
 W_{[1,j]}(\lambda)\left(%
\begin{array}{c}
  0 \\
  1 \\
\end{array}%
\right) &= W_{[1,j-1]}(\lambda) W_j(\lambda)\left(%
\begin{array}{c}
  0 \\
  1 \\
\end{array}\right)\\&=\begin{pmatrix}
  -\varepsilon_{j-1}a_{j-1} Q_{{j-2}}(\lambda) & -Q_{{j-1}}(\lambda) \\
  \varepsilon_{j-1}a_{j-1} P_{{j-2}}(\lambda) & P_{j-1}(\lambda) \\
\end{pmatrix}
\left(%
\begin{array}{c}
  -\varepsilon_j/a_j \\
  p_j(\lambda)/a_j \\
\end{array}%
\right)\\&=\frac{1}{a_j}\begin{pmatrix}
  \varepsilon_{j-1}\varepsilon_ja_{j-1} Q_{{j-2}}(\lambda) -p_j(\lambda)Q_{{j-1}}(\lambda) \\
  -\varepsilon_{j-1}\varepsilon_ja_{j-1} P_{{j-2}}(\lambda) +p_j
  P_{j-1}(\lambda)
\end{pmatrix}.
\end{split}
\]
Due to the difference equation~\eqref{DifEq}
\begin{equation}\label{eq:W01}
    W_{[1,j]}(\lambda)\left(%
\begin{array}{c}
  0 \\
  1 \\
\end{array}\right)=\begin{pmatrix}
  -Q_{{j}}(\lambda) \\
   P_{j}(\lambda) \\
\end{pmatrix}.
\end{equation}
Hence one obtains
\begin{equation}\label{eq:W10}
    W_{[1,j]}(\lambda)\left(%
\begin{array}{c}
  1 \\
  0 \\
\end{array}\right)=W_{[1,j-1]}(\lambda)
\left(%
\begin{array}{c}
  0 \\
\varepsilon_ja_{j} \\
\end{array}\right)
=\varepsilon_ja_{j}
\begin{pmatrix}
  -Q_{{j-1}}(\lambda) \\
   P_{j-1}(\lambda) \\
\end{pmatrix}.
\end{equation}
 The formulas \eqref{eq:W01}-\eqref{eq:W10} prove \eqref{2.26}.
   \end{proof}
\begin{rem}
The recursion algorithm for Nevanlinna functions is well known (see
for example~\cite{Ach61}). The formula~\eqref{2.26} for the
resolvent matrix can be found in~\cite{Kr67}, were truncated moment
problems were studied. The operator approach to such problems was
presented in~\cite{DM91}, \cite{DM95}. In the indefinite case this
algorithm was studied by M. Derevyagin in~\cite{Derev03}, formulas
\eqref{2.26} for the matrix $W(\lambda)$ and the statement of the
theorem for the nondegenerate even moment problem were proven
in~\cite{DD07}. The linear fractional transformations similar to
$\cT_j$ (so-called Schur transform) has been studied by D.~Alpay,
A.~Dijksma and H.~Langer in~\cite{ADL04}, \cite{ADL07}.
   \end{rem}

\section{
Description of solutions} In this section we find a solvability
criterion and describe the set of solutions of the problems
${MP}_\kappa({\bf s},\ell)$ and ${IP}_\kappa({\bf s},\ell)$ in the
general setting. As well as in the case of basic problem we will
distinguish non-degenerate problems and two types of degenerate
problems:
\begin{enumerate}
   \item[(A)] $\rank S_n=n_N=\rank S_{n_N-1}$;
    \item[(B)] $\rank S_n>n_N=\rank S_{n_N-1}$.
\end{enumerate}
\subsection{Non-degenerate moment problem.}
   \begin{thm} \label{thm:5.1}
   Let $({\bf s},\ell)$ be a sequence of real numbers such that $\det S_n\ne
   0$,
Then the moment problem ${MP}_\kappa({\bf s},\ell)$ is solvable if
and only if
\[
\kappa\ge\nu_-.
\]
The sets $\cM_\kappa({\bf s},\ell)$ and $\cI_\kappa({\bf s},\ell)$
are parametrized by the formula~
\begin{equation}\label{2.27B}
    \varphi(\lambda)=\frac{w_{11}^{(N)}(\lambda)\tau(\lambda)+w_{12}^{(N)}(\lambda)}
    {w_{21}^{(N)}(\lambda)\tau(\lambda)+w_{22}^{(N)}(\lambda)},
\end{equation}
where in the even case
\[
\varphi\in\cM_\kappa({\bf s},\ell)\Leftrightarrow \tau \in
\mathbf{N}_{\kappa-\nu_-}\mbox{ and satisfies   (E) };
\]
and in the odd case
\[
\varphi\in\cI_\kappa({\bf s},\ell)\Leftrightarrow \tau \in
\mathbf{N}_{\kappa-\nu_-}\mbox{ and satisfies   (O) };
\]
\[
\varphi\in\cM_\kappa({\bf s},\ell)\Leftrightarrow \tau \in
\mathbf{N}_{\kappa-\nu_-,1}\mbox{ and satisfies   (O) }.
\]
   \end{thm}
   \begin{proof}
The proof is obtained by compilation of Theorem~\ref{thm:2.2},
Theorem~\ref{thm:4.6} and Proposition~\ref{derev}.
   \end{proof}

This result seems to be new even for the odd Hamburger moment
problem.
\begin{cor}\label{cor:5.1}
    Let $s_0,s_1,\dots,s_{2n+1}$ be real numbers, such that $S_n>0$
    and $\det S_n\ne 0$. Then the moment problem~\eqref{0.1} with $\ell=2n+1$ is
    solvable and the formula~\eqref{0.4} describes the set of
    solutions of~\eqref{0.1} when $\tau$ is ranging over the class
    $\mathbf{N}_{0,1}$ and satisfies the condition~(O).
    Moreover, $\varphi\in\cI_0({\bf s},2n+1)$
    if and only if $\tau$ belongs to $\mathbf{N}_{0}$
    and satisfies (O).
\end{cor}

The following example shows the importance of the condition
$\tau\in\mathbf{N}_{0,1}$ in Corollary~\ref{cor:5.1}.
\begin{ex}\label{ex:5.1}
Let $s_0=1$,$s_1=0$. Then $p(\lambda)=\lambda$ and the set of
    solutions of the problem $\cI_0({\bf
s},1)$  is described by
    \[
\varphi(\lambda)=\frac{-1}{p(\lambda)+\tau(\lambda)},
    \]
where $\tau\in\mathbf{N}_{0}$ and $\tau(\lambda)=o(1)$  as
$\lambda\hat\to\infty$. The function $\tau(\lambda)=-\frac{1}{i\ln
(1+\sqrt{\lambda})}$ belongs to the class $\mathbf{N}_{0}$ and
satisfies the condition $\tau(\lambda)=o(1)$  as
$\lambda\hat\to\infty$. Therefore,
\[
\varphi(\lambda)=\frac{-i\ln (1+\sqrt{\lambda})}{i\lambda\ln
(1+\sqrt{\lambda})-1}
\]
is a solution of the problem $IP_0({\bf s},1)$
\[
\varphi(\lambda)=\frac{-1}{\lambda}+o\left(\frac{1}{\lambda^2}\right)
\mbox{ as } \lambda\hat\to\infty.
\]
However, $\tau\not\in \mathbf{N}_{0,1}$ (see~\cite{HLS95}) and,
hence, $\varphi$ is not a solution of the moment
problem~\eqref{0.1}.
\end{ex}
\subsection{Degenerate moment problem. Case (A)}
In Theorem~\ref{thm:3.2} solvability criteria for degenerate moment
problems with minimal negative signature $\kappa=\nu_-(S_n)$ are
given. We start with two lemmas.
   \begin{lem}\label{lem:3.2}
    Let $({\bf s},2n)$ be a sequence of real numbers such that $\det
    S_n=0$.
Then the Hankel rank $n_N$ of the sequence $({\bf s},{2n})$
coincides with the largest normal index $n_N$ of the Hankel matrix
$S_n$.
\end{lem}
\begin{proof}
 By Frobenius Theorem (see~\cite[Lemma X.10.1]{Ga}), if $r$
is the smallest integer $r$ $(0\le r\le n)$, such that
~\eqref{eq:vect} holds, then $\mbox{det }S_{r-1}\ne 0$. Hence $r$ is
the normal index of $S_n$. Moreover, $r$ is the largest normal index
of $S_n$, since the vectors $( s_j,\dots,s_{j+n})^\top $, $(0\le
j\le n_N) $ in~\eqref{eq:vect} are linearly dependent. This implies
that $\det S_n=0$ for all $j\ge r$.
\end{proof}
   \begin{lem}\label{lem:5.4}
    Let $({\bf s},2n)$ be a sequence of real numbers such that $\det
    S_n=0$, let $n_N$ be the largest normal index of the Hankel matrix
$S_n$ and let $S_n$ admit a Hankel extension $S_{n+1}$, such that
$\nu_-(S_{n+1})=\nu_-(S_{n})$. Then there are real numbers
$\alpha_0,\dots,\alpha_{n_N-1}$, such that
   \begin{equation} \label{3.2}
s_{j} =\alpha_0s_{j-n_N}+\dots +\alpha_{n_N-1}s_{j-1}\quad (n_N\le
j\le 2n+1);
  \end{equation}
\begin{equation} \label{3.2b}
s_{2n}\ge\alpha_0s_{2n-n_N+2}+\dots +\alpha_{n_N-1}s_{2n+1}.
\end{equation}
   \end{lem}
   \begin{proof}
    Let us set $v_j=( s_j,\dots,s_{j+n})^\top$, $(0\le j\le n+1) $.
Since $\nu_-(S_n)=\nu_-(S_{n+1})$, it follows from
Lemma~\ref{lem:3.1} that
\[
  v_{n+1}\in\mbox{span }(v_0,\dots,v_{n}),
\]
i.e. there is $c\in\dC^{n+1}$ such that
   \begin{equation} \label{3.3}
   v_{n+1}=S_{n}c.
   \end{equation}
By Lemma~\ref{lem:3.2} there are real numbers
$\alpha_0,\dots,\alpha_{r-1}$, such that
\begin{equation}\label{eq:vect2}
 v_r= \begin{pmatrix}
      s_r \\
    \vdots \\
    s_{r+n}\\
\end{pmatrix}=
\begin{pmatrix}
      s_0 & \dots & s_{n} \\
      \vdots & \adots & \vdots \\
    s_{n} & \dots & s_{2n}
\end{pmatrix} \wh\alpha,\quad\mbox{where } \wh\alpha:=
\begin{pmatrix}
     \alpha_0 \\
    \dots \\
    \alpha_{r-1}\\
   0_{(n-r+1)\times 1}
\end{pmatrix}
   \end{equation}
This  together with~\eqref{3.3} implies, in particular, that
   \begin{equation} \label{3.5}
     s_{r+n+1}=c^*v_r=c^*S_{n}\wh\alpha =v_{n+1}^*\wh\alpha
   =\begin{pmatrix}
        s_{n+1} & \dots & s_{2n+1}
\end{pmatrix}\wh\alpha.
   \end{equation}
Denote by $V$ the $(n+1)\times(n+1)$ forward shift matrix
$V=(\delta_{i,j+1})_{i,j=1}^{n+1}$. Then~\eqref{3.2} and~\eqref{3.5}
imply
   \begin{equation*}
    v_{r+1}= \begin{pmatrix}
      s_{r+1} \\
    \vdots \\
    s_{r+n+1}\\
\end{pmatrix}=
 \begin{pmatrix}
      s_1 & \dots & s_{n+1} \\
   \vdots & \adots & \vdots \\
    s_{n+1} & \dots & s_{2n+1}
\end{pmatrix} \wh\alpha=S_{n}V\wh\alpha
   \end{equation*}
Iterating these calculations one obtains for $(0\le)j\le n$
  \begin{equation}\label{eq:5.7}
    v_{j+1}= \begin{pmatrix}
      s_{j+1} \\
    \vdots \\
    s_{j+n+1}\\
\end{pmatrix}=
\begin{pmatrix}
      s_1 & \dots & s_{n+1} \\
      \vdots & \adots & \vdots \\
    s_{n+1} & \dots & s_{2n+1}
\end{pmatrix} V^{j-r}\wh\alpha=S_{n}V^{j-r+1}\wh\alpha,
 \end{equation}
 which proves~\eqref{3.2}.
Setting in \eqref{eq:5.7} $j=n$ one obtains
    \[
v_{n+1}= S_{n}V^{n-r+1}\wh\alpha .
\]
Then it follows from Lemma~\ref{lem:3.1} that
   \begin{equation*}
   \begin{split}
  s_{2n+2} &\ge (V^{n-r+1}\wh\alpha)^*S_{n}V^{n-r+1}\wh\alpha\\
         &=(V^{n-r+1}\wh\alpha)^*\begin{pmatrix}
      s_{n+1} \\
    \vdots \\
    s_{2n+1}\\
\end{pmatrix}\\
&= \alpha_0s_{2n-r+2}+\dots +\alpha_{r-1}s_{2n+1}.
\end{split}
   \end{equation*}
Hence~~\eqref{3.2b} holds and this completes the proof.
   \end{proof}
This motivates the following definition which in the definite case
was used in~\cite{CF91}.
   \begin{defn}\label{def:3.1}
   A sequence $({\bf s},\ell)=\{s_j\}_{j=0}^{\ell}$ with the Hankel rank
$r=\mbox{rank }{\bf s}$ is called {\it recursively generated}, if
there exist numbers $\alpha_0,\dots,\alpha_{r-1}$, such that
   \begin{equation} \label{3.1}
s_{j} =\alpha_0s_{j-r}+\dots +\alpha_{r-1}s_{j-1}\quad (r\le j\le
\ell).
   \end{equation}
   \end{defn}

   \begin{thm} \label{thm:3.1}
   Let $({\bf s},\ell)$ be a sequence of real numbers such that $\det
   S_n=   0$, $n=[\ell/2]$,  let $n_1<\dots<n_N$ be all normal indices of the degenerate Hankel
matrix $S_n$, let $({\bf s}^{(N)},\ell-2n_N)$ be a sequence of
induced moments determined by successive application of~\eqref{2.7},
and let $\kappa=\nu_-(S_{n})$, $\kappa_N=\nu_-(S_{n_N-1})$. In the
case when $\ell=2n$ is even the following statements are equivalent:
   \begin{enumerate}
\item[(i)]
 The moment problem $\cM_\kappa({\bf s},\ell)$ is solvable;
\item[(ii)] The moment problem $\cM_{0}({\bf s}^{(N)},\ell-2n_N)$ is
solvable;
\item[(iii)]
$s_0^{(N)}=\dots=s_{\ell-2n_N}^{(N)}=0$;
\item[(iv)]
$S_n$ admits a
Hankel extension $S_{n+1}$ such that
$\nu_-(S_{n+1})=\nu_-(S_{n})$;
\item[(v)]
$({\bf s},\ell)$ is recursively generated;
\item[(vi)] $\mbox{rank }S_n=n_N$;
   \end{enumerate}
 If $\ell=2n+1$ is odd, then
\[
(i)\Leftrightarrow(ii)\Leftrightarrow (iii)\Leftrightarrow
(iv')\Leftrightarrow(v)\Leftrightarrow(vi)
\]
where (iv') and (vi') take the form:
   \begin{enumerate}
\item[(iv')]
there exists a real number $s_{2n+2}$, such that
$\nu_-(S_{n+1})=\nu_-(S_{n})$;
\item[(vi')] $\mbox{rank }S_n=n_N$ and $s_{\ell-2n_N}^{(N)}=0$.
   \end{enumerate}
If one of the above conditions  holds, then $\kappa=\kappa_N$ and
the moment problem $\cM_\kappa({\bf s},\ell)$ has the unique
solution given by
\begin{equation}\label{A0}
     \varphi(\lambda)=-\frac{Q_N(\lambda)}{P_N(\lambda)}.
\end{equation}
 \end{thm}
   \begin{proof} {\bf Even case.}
     (i) $\Rightarrow$ (iii) If the moment problem
     $MP_\kappa({\bf s},2n)$ is solvable then by
     Theorem~\ref{thm:2.2} the problem $MP_{\kappa-\kappa_N}({\bf
     s}^{(N)},2(n-n_N))$ is also solvable. If (iii) is not in force
     then it follows from Proposition~\ref{prop:3.3} that
\[
  \kappa-\kappa_N\ge
\nu_-(S_{n-n_N}^{(N)})+\nu_0(S_{n-n_N}^{(N)}).
\]
In view of Corollary~\ref{cor:2.3}
\begin{equation}\label{MinimS}
\nu_-(S_{n-n_N}^{(N)})=\nu_-(S_{n})-\nu_-(S_{n_N-1})=\kappa-\kappa_N.
\end{equation}
Therefore,
\[
\nu_-(S_{n-n_N}^{(N)})\ge
\nu_-(S_{n-n_N}^{(N)})+\nu_0(S_{n-n_N}^{(N)}),
\]
which implies $\nu_0(S_{n-n_N}^{(N)})=0$. But by the same
Corollary~\ref{cor:2.3} $\nu_0(S_{n-n_N}^{(N)})=\nu_0(S_{n})\ne 0$.

(iii) $\Rightarrow$ (ii) If (iii) holds then
$\nu_-(S_{n-n_N}^{(N)})=0$ and by Theorem~\ref{een}
$\varphi_N(\lambda)\equiv 0$ is the unique solution of the problem
$MP_0({\bf s}^{(N)},2(n-n_N))$. The equality $\kappa=\kappa_N$ is
implied by~\eqref{MinimS}.

(ii) $\Rightarrow$ (i) This follows from Theorem~\ref{thm:2.2}.

(iii) $\Rightarrow$ (iv) If (iii) holds then in view of
Theorem~\ref{thm:2.2} and Theorem~\ref{thm:4.6} the unique solution
of the problem $MP_\kappa({\bf s},2n)$ is given by
\[
  \varphi(\lambda)=-\frac{Q_N(\lambda)}{P_N(\lambda)}\in {\bf N}_\kappa.
\]
Since $\varphi$ is rational of $\mbox{rank }\varphi=n_N$ it admits
the asymptotic expansion~\eqref{0.2} for every $n$, in particular,
there exist $s_{2n+1},\,{s_{2n+2}}\in\dR$ such that
\[
\varphi(\lambda)=
-\frac{s_{0}}{\lambda}-\frac{s_{1}}{\lambda^{2}}-\dots
-\frac{s_{2n+2}}{\lambda^{2n+3}}
+o\left(\frac{1}{\lambda^{2n+3}}\right)
\quad(\lambda\widehat{\rightarrow }\infty)
\]
and the corresponding Hankel matrix $S_{n+1}$ has the same rank as
$S_n$. This implies that
 $\nu_-(S_{n+1})=\nu_-(S_n)$.

(iv) $\Rightarrow$ (v) Assume that $S_n$ admits a Hankel extension
$S_{n+1}$ such that $\nu_-(S_{n+1})=\nu_-(S_{n})$. Then by
Lemma~\ref{lem:3.2} there are
   $\alpha_0,\dots,\alpha_{n_N-1}$, such that
\begin{equation}\label{eq:5.8}
s_{j} =\alpha_0s_{j-n_N}+\dots +\alpha_{n_N-1}s_{j-1}\quad (n_N\le
j\le 2n+1).
\end{equation}
By Definition~\ref{def:3.1} this means that $({\bf s},2n)$ is
recursively generated.

(v) $\Rightarrow$ (vi) The condition (v) implies that
\begin{equation}\label{eq:5.9}
\begin{pmatrix}
      s_j \\
    \vdots \\
    s_{j+n}\\
\end{pmatrix}\in\mbox{span }\left\{\begin{pmatrix}
      s_0 \\
    \vdots \\
    s_{n}\\
\end{pmatrix},\dots
\begin{pmatrix}
      s_{n_N-1} \\
    \vdots \\
    s_{n_N-1+n}\\
\end{pmatrix}\right\}
\end{equation}
for all $j\le n $, and therefore, $\mbox{rank }S_n=n_N$.

(vi) $\Rightarrow$ (iii) If $\mbox{rank }S_n=n_N$ then
$\nu_0(S_n)=n-n_N+1$. By Lemma~\ref{lem:2.9}
\[
 \nu_0(S_{n-n_N}^{(N)})=\nu_0(S_n)=n-n_N+1
\]
and, hence, $s_j^{(N)}=0$ for all $j$ such that $0\le j\le
2(n-n_N)$.

{\bf Odd case.} In the odd case the proof of the equivalences
$(i)\Leftrightarrow(ii)\Leftrightarrow (iii)\Rightarrow (iv')$ is
pretty much the same.

     (iv') $\Rightarrow$ (v). Assume that there exists
$s_{2n+2}$ such that $\nu_-(S_{n+1})=\nu_-(S_{n})$. Then by
Lemma~\ref{lem:3.2} there are $\alpha_0,\dots,\alpha_{n_N-1}$, such
that~\eqref{eq:5.8} holds for all for all $(n_N\le)j\le 2n+1$. This
implies,  that the sequence $({\bf s},2n+1)$ is recursively
generated.

(v) $\Rightarrow$ (vi') The statement (v) implies that
\eqref{eq:5.9} holds for all $j\le n+1 $ and hence  there exist
$\beta_1,\dots, \beta_{n_N}$, such that
\[
\begin{pmatrix}
      s_{n+1}\\
    \vdots \\
    s_{2n+1}\\
\end{pmatrix}=\beta_1\begin{pmatrix}
      s_0 \\
    \vdots \\
    s_{n}\\
\end{pmatrix}+\dots+
\beta_{n_N}\begin{pmatrix}
      s_{n_N-1} \\
    \vdots \\
    s_{n_N-1+n}\\
\end{pmatrix}.
\]
Therefore, $\mbox{rank }S_n=n_N$. Let us set
\[
    s_{2n+2}=\beta_1 s_{n+1}+\dots+\beta_{n_N} s_{n+n_N}.
\]
Then $\mbox{rank }S_{n+1}=\mbox{rank }S_n=n_N$ and by (iii) we
obtain
\[
s_{0}^{(N)}=\dots=s_{2n+1-2n_N}^{(N)}=s_{2n+2-2n_N}^{(N)}=0.
\]

 (vi') $\Rightarrow$ (iii) If $\mbox{rank }S_n=n_N$ then it
was shown above that $s_j^{(N)}=0$ for all $j\le 2(n-n_N)$. Now
(iii) holds since $s_{2n+1-2n_N}^{(N)}=0$.
  \end{proof}
\begin{rem}
     If $\varphi$ is a
rational function of degree $r$ and $\varphi$ has the asymptotic
expansion~\eqref{0.2}, then for $n\ge r-1$, by Kronecker theorem
$\rank S_n=r$  and $\varphi\in \mathbf{N}_{\kappa}$, where
$\kappa=\nu_-(S_n)$ (see~\cite[Theorem~16.11.9]{Ga}). Then
by~Theorem~\ref{thm:3.1} the problem $IP_\kappa({\bf s},2n)$ has a
unique solution.

Now, let $\psi\in \mathbf{N}_{\kappa}$ be such that
\[
\psi(\lambda)=\varphi(\lambda)+o\left(\frac{1}{\lambda^{2r+1}}\right)
\mbox{ as } \lambda \wh \to \infty.
\]
Then both $\varphi$ and $\psi$ are solutions of the problem
$IP_\kappa({\bf s},2n)$ and hence
$\psi(\lambda)\equiv\varphi(\lambda)$. This proves the rigidity
result for generalized Nevanlinna functions obtained
in~\cite{ADLRSh10} and proved originally by Burns and Krantz  for
functions from the Schur class, \cite{BKr94}.
\end{rem}
In the next theorem we describe solutions of the problems
$MP_\kappa({\bf s},\ell)$ and $MI_\kappa({\bf s},\ell)$ in the case
where the rank of the Hankel matrix $S_n$ coincides with the Hankel
rank of the sequence.
\begin{thm} \label{thm:3.2}
Let $({\bf s},\ell)$ be a sequence of real numbers such that $\det
S_n=0$ ($n=[\ell/2]$) and let
\[
\rank S_{n}=n_N. 
\]
 Then the problems
$MP_\kappa({\bf s},\ell)$ and $MI_\kappa({\bf s},\ell)$ are solvable
if and only if:
   \begin{enumerate} \item[(i)] either $\kappa=\nu_-(S_n)$ and 
   the equivalent conditions of Theorem~\ref{thm:3.1} are satisfied;
\item[(ii)] or $\kappa\ge\nu_-(S_n)+\nu_0(S_n)$.
\end{enumerate}
If $\kappa\ge\nu_-(S_n)+\nu_0(S_n)$,  $W_{[1,N]}(\lambda)$ is given
by~\eqref{2.26}, and
\begin{equation}\label{eq:WA}
    W(\lambda)=W_{[1,N]}(\lambda)\left(%
\begin{array}{cc}
  1 & 0 \\
  0 & \lambda^{2\nu_0} \\
\end{array}%
\right)
\end{equation}
then  the formula 
\begin{equation}\label{eq:A1}
\varphi(\lambda)=\cT_{W(\lambda)}[\tau(\lambda)],
\end{equation}
describes the sets $\cM_\kappa({\bf s},\ell)$ and $\cI_\kappa({\bf
s},\ell)$ as follows: in the even case
\[
\varphi\in\cM_\kappa({\bf s},\ell)\Leftrightarrow \tau \in
\mathbf{N}_{\kappa-\nu}\mbox{ and satisfies   (E) };
\]
and in the odd case
\[
\varphi\in\cI_\kappa({\bf s},\ell)\Leftrightarrow \tau+s_{\ell-2
n_N}^{(N)} \in \mathbf{N}_{\kappa-\nu}\mbox{ and satisfies   (O) };
\]
\[
\varphi\in\cM_\kappa({\bf s},\ell)\Leftrightarrow \tau+s_{\ell-2
n_N}^{(N)} \in \mathbf{N}_{\kappa-\nu,1}\mbox{ and satisfies   (O)
}.
\]
\end{thm}
   \begin{proof}
   It follows from Theorem~\ref{thm:2.2} that the problem
$MP_\kappa({\bf s},\ell)$ is solvable if and only if the basic
problem $MP_{\kappa-\nu_-(S_{n_N-1})}({\bf s}^{(N)},\ell-2n_N)$ is
solvable. Since $\rank S_{n}=\rank S_{n_N-1}$, then $\nu_\pm
(S_{n})=\nu_\pm( S_{n_N-1})$, and one obtains from
Corollary~\ref{cor:2.3} that $\rank S_{n-n_N}^{(N)}=0$. Hence the
basic problem $\cM_{\kappa-\nu_-}({\bf s}^{(N)},2(n-n_N))$ is of
type (A). The rest of the statements are immediate from
Lemma~\ref{prop:3.2}.
   \end{proof}

 \subsection{Degenerate moment problem. Case (B)}
In this subsection we give solvability criteria and describe
solutions of the problems $MP_\kappa({\bf s},\ell)$ and
$MI_\kappa({\bf s},\ell)$ in the case where the rank of the Hankel
matrix $S_n$ is greater then the Hankel rank of the sequence $({\bf
s},\ell)$.
\begin{thm} \label{thm:5.6}
Let $({\bf s},\ell)$ be a sequence of real numbers such that $\det
S_n=0$ ($n=[\ell/2]$) and let
\[
\rank S_{n}>n_N,
\]
Then the moment problem $\cM_\kappa({\bf s},\ell)$ is solvable if
and only if~\eqref{eq:SolvCond} holds.

Let $\nu_-:=\nu_-(S_n)$, $\nu_0:=\nu_0(S_n)$,  let $\nu$ be defined
by~\eqref{eq:nu}, let $({\bf s}^{(N)},\ell-2n_N)$ be the sequence of
induced moments after $N$ steps, let the integer $m$ be defined by
\[
  s_j^{(N)}=0 \mbox{ for }j<m;\quad \varepsilon:=s_m^{(N)}\ne 0,
\]
let the polynomial $\wh p$ be defined by the formulas~\eqref{eq:4.3}
and~\eqref{4.5b}, where $n:=n-n_N$, let the matrix-valued function
$W_{[1,N]}(\lambda)$ be given by~\eqref{2.26},  and finally let
\[
    W(\lambda)=W_{[1,N]}(\lambda)\left(%
\begin{array}{cc}
  0 & -\wh\varepsilon \\
  \lambda^{2\nu_0}\wh\varepsilon & \lambda^{2\nu_0}\wh p(\lambda) \\
\end{array}%
\right)
\]
If $\kappa$ satisfies~\eqref{eq:SolvCond}, then  the
formula~\eqref{eq:A1}, describes the sets $\cM_\kappa({\bf s},\ell)$
and $\cI_\kappa({\bf s},\ell)$ as follows: in the even case
\[
\varphi\in\cM_\kappa({\bf s},\ell)\Leftrightarrow \tau \in
\mathbf{N}_{\kappa-\nu-\nu_-}\mbox{ and satisfies   (E) }.
\]
and in the odd case
\[
\varphi\in\cI_\kappa({\bf s},\ell)\Leftrightarrow \tau+s_{\ell-2
n_N}^{(N)} \in \mathbf{N}_{\kappa-\nu-\nu_-}\mbox{ and satisfies (O)
};
\]
\[
\varphi\in\cM_\kappa({\bf s},\ell)\Leftrightarrow \tau+s_{\ell-2
n_N}^{(N)} \in \mathbf{N}_{\kappa-\nu-\nu_-,1}\mbox{ and satisfies
(O) }.
\]
\end{thm}
     \begin{proof}
   It follows from Theorem~\ref{thm:2.2} that the problem
$MP_\kappa({\bf s},\ell)$ is solvable if and only if the basic
problem $MP_{\kappa-\nu_-(S_{n_N-1})}({\bf s}^{(N)},\ell-2n_N)$ of
type (B) is solvable. By~Proposition~\ref{prop:3.3} the latter
problem is solvable if and only if
\[
\kappa-\nu_-(S_{n_N-1})\ge
\nu_0(S_{n-n_N}^{(N)})+\nu_-(S_{n-n_N}^{(N)}).
\]
Due to Corollary~\ref{cor:2.3}
\[
\nu_0(S_{n-n_N}^{(N)})=\nu_0(S_{n}),\quad
\nu_-(S_{n-n_N}^{(N)})=\nu_-(S_{n})-\nu_-(S_{n_N-1}).
\]
These proves the criterion~\eqref{eq:SolvCond}. The statement for
the problem $IP_\kappa({\bf s},\ell)$ is proved by the same
reasonings.

The second part of the proof is implied by Theorem~\ref{thm:2.2},
Theorem~\ref{thm:4.6} and Theorem~\ref{prop:4.1}.
   \end{proof}

\appendix

\section{Some lemmas on matrices}

\subsection{Some lemmas for block matrices}

The verification of the first lemma is left to the reader.

\begin{lem}
The inverse of the invertible block matrix $\wt A$ of the form
\[
 \wt A
=\begin{pmatrix}
   0            & 0        & A_{13} \\
   0            &  A_{22}  & A_{23}\\
    A_{13}^*      &  A_{23}^*& A_{33}\\
\end{pmatrix}
\]
is given by
\[
\wt A^{-1}=
\begin{pmatrix}
-A_{13}^{-*}(A_{33}-A_{23}^*A_{22}^{-1}A_{23})A_{13}^{-1}
                        & -A_{13}^{-*}A_{23}^*A_{22}^{-1}  & A_{13}^{-*} \\
-A_{22}^{-1}A_{23}A_{13}^{-1}              &  A_{22}^{-1}  & 0           \\
 A_{13}^{-1}                               &  0            & 0\\
\end{pmatrix}
\]
\end{lem}

The following lemma extends a well-known result for nonnegative
block matrices.

\begin{lem}\label{lem:3.1}
Let $\wt A$ be a Hermitian matrix of the form
\[
\wt A=
\begin{pmatrix}
    A  & B\\
    B^*& C\\
\end{pmatrix},\quad A\in\dC^{n\times n}, \,\, C\in\dC^{m\times m}
\]
such that $\nu_-(\wt A)=\nu_-(A)$. Then
\begin{enumerate}\def\labelenumi{(\roman{enumi})}
\item $\ran B\subseteq\ran A$;

\item if $Bh=Ag$ for some $h\in\dC^m$, $g\in\dC^n$, then
$h^*Ch\ge g^*Ag$.
\end{enumerate}
\end{lem}

\begin{proof}
First assume that $A$ is invertible. Then the identity
   \begin{equation*}
\wt A=\begin{pmatrix}
    A  & B\\
    B^*& C\\
\end{pmatrix}= \begin{pmatrix}
    I  & 0\\
    B^*A^{-1}& I\\
\end{pmatrix}\begin{pmatrix}
    A  & 0\\
    0& C-B^*A^{-1}B\\
\end{pmatrix}\begin{pmatrix}
    I  & A^{-1}B\\
    0& I\\
\end{pmatrix},
   \end{equation*}
and the assumption $\nu_-(\wt A)=\nu_-(A)$ shows that $C-B^*A^{-1}B
\ge 0$. If $A$ is not invertible, then there is a block
decomposition
\begin{equation*}
 A=\begin{pmatrix}
    A_0  & 0\\
    0& 0\\
\end{pmatrix}   ,\quad
B=\begin{pmatrix}
    B_0 \\
    B_1\\
\end{pmatrix} ,
   \end{equation*}
where $A_0$ is invertible. Note that $\nu_-(A)=\nu_-(A_0)$. Now the
previous statement can be applied to the block matrix
\[
  \wt A=\begin{pmatrix}
    A_0  & 0 & B_0\\
      0  & 0 & B_1\\
    B_0^*&B_1^*& C\\
\end{pmatrix}
\]
with the invertible matrix $A_0$ in the left upper corner, since
$\nu_-(\wt A)=\nu_-(A_0)$. Thus, one concludes that
\[
 \begin{pmatrix}
   0 & B_1\\
B_1^*& C\\
\end{pmatrix}
-\begin{pmatrix}
      0 \\
    B_0^*\\
\end{pmatrix}A_0^{-1}\begin{pmatrix}
 0 & B_0\\
\end{pmatrix}=\begin{pmatrix}
     0 & B_1\\
  B_1^*& C-B_0^*A_0^{-1}B_0\\
\end{pmatrix}\ge 0.
\]
This implies $B_1=0$ and
\begin{equation}\label{ineq00}
C-B_0^*A_0^{-1}B_0\ge 0.
\end{equation}
The identity $B_1=0$ means that $\ran B\subseteq\ran A_0=\ran A$
proving the range inclusion in (i). On the other hand, it follows
from \eqref{ineq00} that the vectors $h$ and $g_0$ for which
$B_0h=A_0g_0$ satisfy the inequality
\[
h^*Ch\ge h^*B_0^*A_0^{-1}B_0h=g_0^*A_0g_0=g^*Ag.
\]
Now (ii) is implied by this inequality.
\end{proof}

\subsection{Some results for Hankel matrices}

   \begin{lem} \label{lem:2.9}
Let $S_{n}$ be a Hankel matrix, let $n_1$ be the first normal index
of $S_{n}$, and let the polynomial
$p_1(\lambda)=p_{n_1}^{(1)}\lambda^{n_1}+\dots+ p_0^{(1)}$ and the
moment sequence $({\bf s}^{(1)},\ell-2n_1)$ be defined
by~\eqref{2.7}. Then for all $i=\overline{0,n-n_1}$
   \begin{equation} \label{2.9a}
\nu_\pm(S_{i}^{(1)})=\nu_\pm(S_{i+n_1})-\nu_\pm(S_{n_1-1});
 \end{equation}
 \begin{equation} \label{2.9b}
\nu_0(S_{i}^{(1)})=\nu_0(S_{i+n_1}).
  \end{equation}
   \end{lem}
   \begin{proof}
Consider the equation~\eqref{2.7} with $j=2i$.  Multiplying it both
from the left and from the right by the matrix $J_{2i+n_1+2}$ one
obtains the equality
   \begin{equation} \label{2.10}
\begin{pmatrix}
   0            &  \wt A_{12} \\
   \wt A_{12}^* &  \wt A_{22}\\
\end{pmatrix}
\begin{pmatrix}
   B_{11}   &  B_{12} \\
   B_{12}^* &  B_{22}\\
\end{pmatrix}
=\varepsilon_1I_{2i+n_1+2},
   \end{equation}
where
\[
\wt A_{12}=\begin{pmatrix}
   0 & \dots &\dots & \dots& 0      &   s_{n_1-1} \\
    &&&&&\\
\vdots&      &      &\adots& \adots &   \vdots \\
&&&&&\\
0 & \dots    &0     &  s_{n_1-1}    &\dots     &   s_{n_1+i-1} \\
\end{pmatrix}
\in\dC^{(i+1)\times(n_1+i+1)},
\]
\[
\wt A_{22}=S_{n_1+i}\in\dC^{(n_1+i+1)\times(n_1+i+1)},
\]
   \begin{equation} \label{2.11}
B_{11}=-\varepsilon_1a_1^2J_{i+1} S_{i}^{(1)}J_{i+1}\in
\dC^{(i+1)\times(i+1)}
   \end{equation}
and $B_{12}$, $B_{22}$ are some matrices from
$\dC^{(i+1)\times(i+n_1+1)} $ and $\dC^{(i+n_1+1)\times(i+n_1+1)}$,
respectively. Let us decompose the matrices $\wt A_{12}$, $\wt
A_{22}$ as follows
   \begin{equation*}
\wt A_{12}=\begin{pmatrix}
    0_{(i+1)\times n_1}  & A_{13}\\
\end{pmatrix} ,\quad
\begin{array}{c}
  A_{13}=\begin{pmatrix}
            &         &   s_{n_1-1} \\
            &\adots   &     \vdots \\
s_{n_1-1}    &  \dots  &   s_{n_1+i-1} \\
\end{pmatrix}\in\dC^{(i+1)\times (i+1)},\,
\end{array}
   \end{equation*}
   \begin{equation*}
\wt A_{22}=\begin{pmatrix}
    A_{22}  & A_{23}\\
    A_{23}^*& A_{33}\\
\end{pmatrix}
=\begin{pmatrix}
    S_{n_1-1}  & A_{23}\\
    A_{23}^*& A_{33}\\
\end{pmatrix}  ,\quad
\begin{array}{c}
    A_{23}\in\dC^{n_1\times(i+1)},\\
  A_{33}\in\dC^{(i+1)\times(i+1)}. \\
\end{array}
   \end{equation*}
Since the matrices $A_{13}=A_{13}^*$, $A_{22}=S_{n_1-1}$ are
invertible, it follows from Lemma A1 that the matrix
\[
 \wt A
=\begin{pmatrix}
   0            & 0        & A_{13} \\
   0            &  A_{22}  & A_{23}\\
    A_{13}   &  A_{23}^*& A_{33}\\
\end{pmatrix}
\]
is invertible, and the left upper corner $B_{11}$ of the inverse
matrix $\wt A^{-1}$ is given by
   \begin{equation} \label{2.12}
\wt
B_{11}=(A^{-1})_{11}=-A_{13}^{-1}(A_{33}-A_{23}^*A_{22}^{-1}A_{23})A_{13}^{-1}
   \end{equation}
 It follows from~\eqref{2.10}-\eqref{2.12} that
   \begin{equation} \label{2.13}
   a_1^2J_{i+1} S_{i}^{(1)}J_{i+1}=-\varepsilon_1B_{11}=-(\wt
A^{-1})_{11}=
   A_{13}^{-1}(A_{33}-A_{23}^*A_{22}^{-1}A_{23})A_{13}^{-1}
   \end{equation}
 and hence
   \begin{equation} \label{2.14}
\nu_\pm(S_{i}^{(1)})=\nu_\pm(A_{33}-A_{23}^*A_{22}^{-1}A_{23});
 \end{equation}
 \begin{equation} \label{2.15}
\nu_0(S_{i}^{(1)})=\nu_0(A_{33}-A_{23}^*A_{22}^{-1}A_{23}).
  \end{equation}
The numbers of positive (negative, zero) eigenvalues of the Schur
complement $A_{33}-A_{23}^*A_{22}^{-1}A_{23}$ can be calculated by
the formulas
   \begin{equation} \label{2.16}
\nu_\pm(A_{33}-A_{23}^*A_{22}^{-1}A_{23})=\nu_\pm(\wt
A_{22})-\nu_\pm(A_{22});
 \end{equation}
 \begin{equation} \label{2.17}
\nu_0(A_{33}-A_{23}^*A_{22}^{-1}A_{23})=\nu_0(\wt
A_{22})-\nu_0(A_{22}).
  \end{equation}
Since $\wt A_{22}=S_{i+n_1}$ and $A_{22}=S_{n_1-1}$ is invertible
the statements of the Lemma~\ref{lem:2.9} are implied
by~\eqref{2.14}-\eqref{2.17}.
   \end{proof}

\end{document}